\documentclass[10pt,a4paper]{article}
\usepackage[utf8]{inputenc}
\usepackage[T1]{fontenc}
\usepackage{amstext,amsfonts,a4}
\usepackage{amsmath,amssymb,amsthm}
\usepackage{hyperref}
\usepackage{graphicx}
\usepackage[ruled,vlined,algosection]{algorithm2e}
\usepackage{color}
\usepackage{fancybox}
\usepackage{pgf,tikz}
\usepackage{chngcntr}
\counterwithin{figure}{section}

\newcommand*{\inalbe}{\int}
\newcommand*{\fplus}{\genfrac{}{}{0pt}{}{}{+}}
\newcommand*{\fdots}{\genfrac{}{}{0pt}{}{}{\cdots}}

\newtheorem{Thm}{Theorem}[section]
\newtheorem{Lemme}[Thm]{Lemma}
\newtheorem{Cor}[Thm]{Corollary}
\newtheorem{Prop}[Thm]{Proposition}

\newtheorem{Rem}[Thm]{Remark}
\newtheorem{Ex}[Thm]{Example}
\def\downbar#1{
\setbox10=\hbox{$#1$}
   \dimen10=\ht10 \advance\dimen10 by 2.5pt
   \ifdim \dimen10<15pt %equals approximately 0.5cm
      \advance\dimen10 by -0.5pt
      \dimen11=\dimen10
      \advance\dimen10 by 2.5pt
      \lower \dimen11
   \else \lower \ht10 \fi
   \hbox {\hskip 1.5pt \vrule height \dimen10 depth \dp10}\relax}
 \def\upbar#1{
 \setbox10=\hbox{$#1$}
    \dimen10=\ht10 \advance\dimen10 by \dp10 \advance\dimen10 by 2.5pt
    \ifdim \dimen10<15pt %equals approximately 0.5cm
       \advance\dimen10 by 2pt \fi
    \raise 2.5pt \hbox {\hskip -1.5pt \vrule height \dimen10}\relax}
\def\cfr#1#2{
 \downbar{#2} \hskip -1.5pt {\; #1 \; \over \thinspace \  #2}\upbar{#1}}
\DeclareMathOperator{\supp}{supp}

\DeclareMathOperator{\sign}{sign}
\DeclareMathOperator{\capa}{cap}
\DeclareMathOperator{\rank}{rank}

\numberwithin{equation}{section}
\begin{document}
\renewcommand{\thefootnote}{\fnsymbol{footnote}}
\begin{center}
    {\LARGE\bf     On the rational approximation of Markov functions, with applications to the computation of Markov functions of Toeplitz 
    %matricesComputing Markov functions of Toeplitz 
    matrices}
    \\[10pt]
    by Bernhard Beckermann\footnotemark, Joanna Bisch\footnotemark[\value{footnote}],
    \footnotetext{Laboratoire Paul Painlev\'e UMR 8524, D\'epartement de Math\'ematiques, Universit\'e de Lille, F-59655 Villeneuve d'Ascq, France.  The work has been supported in part by the Labex CEMPI (ANR-11-LABX-0007-01). Corresponding author: Bernhard.Beckermann@univ-lille.fr}
    and Robert Luce\footnote[4]{Gurobi Optimization, LLC., 9450 SW Gemini Dr.
    \#90729 Beaverton, Oregon, USA.  luce@gurobi.com}
\end{center}
%	\maketitle
%	\tableofcontents
\begin{abstract}
    We investigate the problem of approximating the matrix function $f(A)$ by $r(A)$, with $f$ a Markov function, $r$ a rational interpolant of $f$, and $A$ a symmetric Toeplitz matrix. In a first step, we obtain a new upper bound for the relative interpolation error $1-r/f$ on the spectral interval of $A$.
    By minimizing this upper bound over all interpolation points, we obtain a new, simple and sharp a priori bound for the relative interpolation error. We then consider three different approaches of representing and computing the rational interpolant $r$. Theoretical and numerical evidence is given that any of these methods for a scalar argument allows to achieve high precision, even in the presence of finite precision arithmetic. We finally investigate the problem of efficiently evaluating $r(A)$, where it turns out that the relative error for a matrix argument is only small if we use a partial fraction decomposition for $r$ following Antoulas and Mayo. An important role is played by a new stopping criterion which ensures to automatically find the degree of $r$ leading to a small error, even in presence of finite precision arithmetic.
\end{abstract}	

\renewcommand{\thefootnote}{\arabic{footnote}}

	\section{Introduction and statement of the results}

The need for computing matrix functions $f(A)$ for some square matrix $A\in \mathbb R^{n\times n}$ and some function being analytic on some neighborhood of the spectrum of $A$ arises in a variety of applications, including network analysis \cite{Benzi20,Estrada10}, signal processing \cite{Shuman13}, machine learning \cite{Stoll20}, and differential equations \cite{Hochbruck10}.
%\footnote{M. Benzi and P. Boito, Matrix functions in network analysis, GAMM Mitteilungen, (2020).To appear}, \footnote{E. Estrada and D. J. Higham, Network properties revealed through matrix functions, SIAM Rev., 52 (2010), pp. 696–714}, signal processing \footnote{D.  I.  Shuman,  S.  K.  Narang,  P.  Frossard,  A.  Ortega,  and  P.  Vandergheynst, The emerging field of signal processing on graphs:  Extending high-dimensional data analysis to networks and other irregular domains, IEEE Signal Process. Mag., 30 (2013), pp. 83–98.}, machine learning \footnote{M.  Stoll, A   literature   survey   of   matrix   methods   for   data   science,   tech.   rep.,   2019. arXiv:1912.07896}, and differential equations\footnote{M.  Hochbruck  and  A.  Ostermann, Exponential  integrators,  Acta  Numerica,  19  (2010), pp. 209–286.}.
We refer the reader to \cite{higham2008} and the references therein for a detailed account on computing matrix functions for various functions $f$. In the present paper we are interested in the particular case of Markov functions, that is, the Cauchy transform of a positive measure $\mu$ with support $\supp(\mu)\subset \mathbb R$, and more precisely
\begin{equation} \label{eq_markov}
 f^{[\mu]}(z) = \inalbe \frac{d\mu(x)}{z-x}, \quad \mbox{with infinite $\supp(\mu)\subset [\alpha,\beta]$ for suitable $-\infty\leq \alpha<\beta<+\infty$.}
\end{equation}
This includes the functions $f^{[\mu]}(z)=\frac{\log(z)}{z-1}$ or $f^{[\mu]}(z)=z^\gamma$ for $\gamma \in (-1,0)$ and $\supp(\mu)=(-\infty,0]$, but also many other elementary functions, see for instance \cite{Henr77}. %\footnote{P. Henrici, Applied  and  Computational  Complex  Analysis,  Vol.  2, John Wiley \& Sons, New York, 1977.}.
In particular, elementary computations show that\footnote{This includes the limiting case $\frac{d\nu}{dx}(x) = \frac{1}{\pi \sqrt{z-\beta}}$ and $f^{[\nu]}(z)=1/\sqrt{z-\beta}$ for $\alpha \to -\infty$.}
\begin{equation} \label{sqrt}
%    f(z)=\frac{\log(z)}{z-1} , \quad
%    f(z)=z^\gamma, \quad
    f^{[\nu]}(z)=\frac{\sqrt{|\alpha|}}{\sqrt{(z-\alpha)(z-\beta)}}, \quad \mbox{with density}
     \quad \frac{d\nu}{dx}(x) =
    \frac{\sqrt{|\alpha|}}{\pi \sqrt{(x-\alpha)(\beta - x)}} \quad \mbox{on $\supp(\nu)=[\alpha,\beta]$.}
\end{equation}
The main reason for restricting ourselves to Markov functions is that many results about best rational approximants and rational interpolants are known, see for instance the first paragraph in \S\ref{sec2} and the references therein. In addition, for evaluating $r(A)$ for a rational function $r$ we can fully exploit the structure of $A$:
if $A$ is a Toeplitz matrix
\begin{equation}\label{toeplitz}
A=\begin{pmatrix}
t_{0} & t_{-1} &\ldots &\ldots & t_{-n+1}\\
t_{1} & t_{0} & t_{-1} & \ldots & t_{-n+2} \\
\vdots & t_{1} & \ddots & \ddots & \vdots \\
\vdots & & \ddots & \ddots & t_{-1} \\
t_{n-1} & \ldots & \ldots & t_{1} & t_{0}
\end{pmatrix}
\end{equation}
then using the concept of displacement rank we just need $\mathcal O(n \log^2(n))$ operations and $\mathcal O(n)$ memory requirements (the hidden constant depending on the degree), and a similar property seems to be true for matrices with hierarchical rank structure, see \S\ref{sec4} and the references therein.

To be more precise, denote by $\mathcal R_{m,n}$ the set of rational functions, with numerator degree $\leq m$, and denominator degree $\leq n$. Given $z_1,...,z_{2m}$ called interpolation points in $\mathbb C\setminus [\alpha,\beta]$, a {\itshape rational interpolant} (also sometimes called multi-point Padé approximant) $r_m^{[\mu]}$ of $f^{[\mu]}$ of type $[m-1|m]$ is a rational function in $\mathcal R_{m-1,m}$ which interpolates $f^{[\mu]}$ at $z_1,...,z_{2m}$ (in the sense of Hermite if some of the interpolation points occur with multiplicity $>1$). For Markov functions it is known that, provided that the non-real interpolation points occur in conjugate pairs, there is one and only one rational interpolant of type $[m-1|m]$, and this interpolant has $m$ simple poles in $(\alpha,\beta)$, and positive residuals. For instance, Padé approximants matching Taylor expansions at $z_1$ are a special case of rational interpolants, with $z_1=...=z_{2m}$. Also, from equi-oscillation we know that a best rational approximant in $\mathcal R_{m-1,m}$ with respect to maximum norm on some interval $[c,d]\subset (\beta,+\infty)$ is a rational interpolant of $f^{[\mu]}$. In contrast, approximating $f(A)b$ by projection on rational Krylov spaces \cite{beckermann09} gives raise to an expression $r(A)b$ where $r\in \mathcal R_{m-1,m}$ has prescribed poles, and satisfies only $m$ interpolation conditions, a so-called Padé type rational approximant.

Inspired by several authors \cite{higham2008}, in the present paper we will approximate  $f^{[\mu]}(A)$ by $r^{[\mu]}_m(A)$ for suitable interpolation points. For instance, the authors in \cite{Higham13} use Padé approximants (represented as convergents of a Stieltjes continued fraction) for fractional powers \eqref{sqrt}, combined with scaling and squaring techniques. The computation of rational interpolants $r_m^{[\mu]}$ of Markov functions $f^{[\mu]}$ is known to be delicate on a computer with finite precision arithmetic: we get instead a rational function $\widetilde r_m$ which might be far from $r^{[\mu]}_m$, depending on how to represent and to compute the interpolant. In addition, on a computer we will obtain a matrix $R$ instead of $\widetilde r_m(A)$, again due to finite precision arithmetic. We are still far from a full understanding how these different errors accumulate, and thus will be interested in the present paper mainly in the case of symmetric matrices $A$. Denote by $\mathbb E \subset (\beta,+\infty)$ a closed set containing all eigenvalues of $A$, for instance the spectral interval spanned by the smallest and largest eigenvalue of $A$. We thus are interested in the three relative errors
\begin{equation} \label{goal}
     \left\| 1 - {r_m^{[\mu]}} / {f^{[\mu]}} \right\|_{L^\infty(\mathbb E)} , \quad
     \left\| 1 - {\widetilde r_m} / {f^{[\mu]}} \right\|_{L^\infty(\mathbb E)} , \quad
     \left\| I - R \, f^{[\mu]}(A)^{-1} \right\| ,
\end{equation}
the first two being upper bounds of $\| I - {r_m^{[\mu]}}(A) {f^{[\mu]}}(A)^{-1} \|$, and $\| I - {\widetilde r_m(A)} {f^{[\mu]}}(A)^{-1} \|$, respectively. These three quantities will be discussed in \S\ref{sec2}, \S\ref{sec3}, and \S\ref{sec_experiences}, respectively.

    Let us highlight the main theoretical contributions of this paper. In Theorem~\ref{thmprincR} we suggest new error bounds in terms of Blaschke products of order $2m$ for the relative\footnote{In most papers in the literature, the authors estimate the absolute interpolation error. However, by considering the relative error we may monitor also the error of rational approximants of functions $f$ being a product of a Markov function and a rational function such as the logarithm or the square root, see Examples~\ref{Ex5.1} and~\ref{Ex5.3}.} interpolation error $\| 1 - {r_m^{[\mu]}} / {f^{[\mu]}} \|_{L^\infty(\mathbb E)}$ on subsets $\mathbb E$ of the real line for quite arbitrary $m,\alpha,\beta,z_1,...,z_{2m}$. It follows that, up to some modest constant, there is a worst case measure for the relative error given by the measure $\nu$ of \eqref{sqrt}. As a consequence of Theorem~\ref{thmprincR}, we derive in Corollary~\ref{cor_posteriori} new residual and a posteriori upper bounds for $\| I - {r_m^{[\mu]}}(A) {f^{[\mu]}}(A)^{-1} \|$ which do not require to know ${f^{[\mu]}}(A)$. Restricting ourselves to intervals $\mathbb E=[c,d]$, this allows us in Remark~\ref{rem_eta} for the Padé case $z_1=...=z_{2m}$ to find an optimal $z_1$, and in Corollary~\ref{cor_a_priori} the quasi optimal interpolation points which minimize our upper bound of Theorem~\ref{thmprincR}. In the latter case, we deduce a very simple a priori error bound of asymptotic form $8\rho^{2m}$ in terms of the logarithmic capacity of the underlying condenser, which is sharp and again seems to be new. Based on these results, we suggest in Remark~\ref{rem_stopping} a stopping criterion allowing to find automatically the $m$ leading to a small interpolation error, even in the presence of rounding errors.

    In Theorem~\ref{thmprinc2} we estimate the absolute interpolation error on the closed unit disk, which is combined with the Faber operator techniques of \cite{beckermann09} in order to construct rational functions $r\in \mathcal R_{m-1,m}$ with an explicit bound  for the error $\| f - r \|_{L^\infty(\mathbb E)}$ for compact and convex sets $\mathbb E$, such as the field of values of $A$ in case where $A$ is not symmetric. Again, optimizing the interpolation points, our bounds improve results of Knizhnerman  on Faber-Padé approximants \cite{Kniz08}.
    Finally, this paper also contains two new results on interpolating Thiele continued fractions: in Theorem~\ref{Thm_Markov} we show that (reciprocal) Markov functions give raise to an interpolating continued fraction with positive parameters, which allows us to show in Theorem~\ref{Thm_finite_precision} the backward stability of positive Thiele fractions, improving \cite[Theorem 4.1]{GM80} of Graves-Morris.

    We conclude this introduction by summarizing the structure of the paper. In the first paragraph of \S\ref{sec2} we recall several results scattered in the literature on upper bounds for rational interpolants and best rational approximants of Markov functions. We then state and prove our new bounds for the interpolation error on subsets of the real line in \S\ref{sec2.1}, and on the unit disk in \S\ref{sec2.2}. In order to monitor the second term in \eqref{goal}, we will discuss in \S\ref{sec3} three different ways of representing and computing $r_m^{[\mu]}$, namely in \S\ref{sec3.1} a partial fraction decomposition, %the poles being eigenvalues of some Loewner matrix pencil, and the residuals from solving the interpolation conditions in the least-square sense.
    in \S\ref{sec3.2} %we consider
    an interpolating barycentric representation of $r_m^{[\mu]}$, and in \S\ref{sec3.3} a Thiele interpolating continued fraction, which generalizes the above-mentioned Stieltjes continued fraction to arbitrary distinct and real interpolation points. %Our main new contributions in this section are Theorem~\ref{Thm_Markov} where we show that (reciprocal) Markov functions give raise to an interpolating continued fraction with positive parameters, which allows us to show in Theorem~\ref{Thm_finite_precision} the backward stability of positive Thiele fractions, improving \cite[Theorem 4.1]{GM80} of Graves-Morris.
    We %finally
    give in Figures~\ref{fig_representation1} and \ref{fig6}--\ref{fig8} numerical evidence that we may reach nearly machine precision for the error $\| 1 - {\widetilde r_m} / {f^{[\mu]}}\|_{L^\infty(\mathbb E)}$ in \eqref{goal} for any of the three representations $\widetilde r_m$ of $r_m^{[\mu]}$, if we use the stopping criterion of Remark~\ref{rem_stopping}.

    In \S\ref{sec4} we provide more information how to evaluate $\widetilde r_m(A)$ for a (symmetric) Toeplitz matrix $A$, using the concept of small displacement rank. In particular, we show in Theorem~\ref{Thm_computation} the above claimed complexity $\mathcal O(n \log^2(n))$ and memory requirements $\mathcal O(n)$. Finally, in \S\ref{sec_experiences} we give numerical experiments, and investigate also possible improvements through a combination with scaling and squaring for particular functions. We conclude that only a representation of $r_m^{[\mu]}$ as a partial fraction decomposition allows to attain small errors, if we want to exploit the Toeplitz structure of $A$.

\section{The error of rational interpolants of Markov functions}\label{sec2}

The aim of this section is to estimate the error of rational interpolants $r_m^{[\mu]}$ of type $[m-1|m]$ of a Markov function $f^{[\mu]}$ of a measure $\mu$ with support in
$[\alpha,\beta]$, both on a real set $\mathbb E \subset \mathbb R\setminus [\alpha,\beta]$ as for instance a real interval in our Theorem~\ref{thmprincR}, and on the unit disk in our Theorem~\ref{thmprinc2}. We are less interested in asymptotic results on the error, and refer the interested reader to the work of Gonchar \cite{Gonch78}  and the book of Stahl and Totik \cite{Tot92} for $m$th root asymptotics, the work of L\'opez Lagomasino \cite{Lago86,Lago87}
%\footnote{G. L\'opez Lagomasino, On the asymptotics of the ratio of orthonormal polynomials and the convergence of multipoint Padé approximants, Math. USSR Sbornik, 56 (1987), 216–229. and G. L\'opez Lagomasino, Szego's Theorem for polynomials orthogonal with respect to varying measures, Orthogonal Polynomials and their Applications, Lect. Notes in Math., Springer-Verlag, 1329 (1986), 255–260.}
on ratio asymptotics, and the work of Stahl \cite{Stahl00}
%\footnote{H. Stahl, Strong asymptotics for orthonormal polynomials with varying weights. Acta Sci. Math. (Szeged) 66 (2000), no. 1-2, 147–192.}
on strong asymptotics, though some of the tools in these papers are also of help for deriving upper bounds. In this paper we want to derive upper bounds of the form $C\rho^{2m}$, where the constants $C,\rho$ only depend on $[\alpha,\beta]$ and the interpolation points, but not on the regularity of $\mu$. Previous work on this subject include error estimates for Padé approximants of Markov functions, see, e.g., the book of Baker and Graves-Morris \cite[Thm 5.2.6 and Thm 5.4.4]{BGM96}. We are only aware of work of Ganelius \cite[Chap. 4]{Gan82} and Braess \cite[Thm~2.1]{Braess87} on upper bounds of the form $C\rho^{2m}$. These authors look for particularly well-chosen interpolation points which allow to make the link with best rational approximants. Also we should mention the more recent work of Knizhnerman \cite[Part.2, Section 3.1]{Kniz08} on Faber Padé approximants, who does not give an explicit value of $C$. Our aim is to improve the constant $C$ in all these findings. Also, we want to prove the claim in \cite{beckermann09} that for rational interpolants with free poles we should get the square of the bounds of \cite{beckermann09} and \cite{Massei21} %\footnote{\textcolor{red}{Stefano Massei \& Leonardo Robol, Rational Krylov for Stieltjes matrix functions: convergence and pole selection, BIT Numerical Mathematics volume 61, pages 237–273 (2021).}}
in terms of (minimal) Blaschke products obtained for rational interpolants with prescribed poles. %Here our Theorem~\ref{thmprincR} treats with the interval case, and Theorem~\ref{thmprincR} discusses the unit disk case.

\subsection{Estimates on the real line}\label{sec2.1}
We start with the interval case, where we allow for more general real or complex conjugate %distinct
interpolation points $z_1,...,z_{2m}\in \mathbb C\setminus [\alpha,\beta]$ and estimate the relative interpolation error. The comparison principle \cite[Lemma V.3.8 and Thm V.3.9]{bra86} allows to relate absolute errors for rational interpolants of Markov functions for two measures $\mu\leq \nu$. The situation is different for relative errors since, as we show in the next theorem, there is (up to a factor 2 or 3) a worst case measure given by the scaled equilibrium measure of the interval $[\alpha,\beta]$ (and thus neither depending on the choice of $\mathbb E$ nor on the interpolation points). We also give upper bounds in terms of Blaschke products.

\begin{Thm}\label{thmprincR}
    Let $-\infty\leq \alpha<\beta<\infty$, and let the Markov functions $f^{[\mu]}$ and $f^{[\nu]}$ be as in \eqref{eq_markov} and \eqref{sqrt}.
    Furthermore, let  $\mathbb E \subset \mathbb R\setminus [\alpha,\beta]$, and consider interpolation points $z_1,...,z_{2m}\in\mathbb C\setminus [\alpha,\beta]$ where we suppose that non-real points only occur in conjugate pairs. We refer to the positive case if the real interpolation points have even multiplicity.\footnote{If $\mathbb E$ is a finite union of closed intervals, it is sufficient to suppose that interpolation points in $Int(\mathbb E)$ only have even multiplicity, and there is an even number of interpolation points in any subinterval of $\mathbb R \setminus Int(\mathbb E)$.} Then for the interpolant $r_m^{[\mu]}$ of type $[m-1|m]$ of $f^{[\mu]}$ we may bound the relative error as follows
	\begin{equation}\label{equaprinc}
	\|\frac{f^{[\mu]}-r_{m}^{[\mu]}}{f^{[\mu]}}\|_{L^{\infty}(\mathbb E)}
	\leqslant
    \left\{\begin{array}{ll}
      2 \|\frac{f^{[\nu]}-r_{m}^{[\nu]}}{f^{[\nu]}}\|_{L^{\infty}(\mathbb E)}
      \leq 4 \eta_{2m} & \mbox{in the positive case,}
      \\
      \| 1- \Bigl(\frac{r_{m}^{[\nu]}}{f^{[\nu]}}\Bigr)^2 \|_{L^{\infty}(\mathbb E)}
      \leq
      4 \frac{\eta_{2m}}{(1-\eta_{2m})^2} & \mbox{in the general case,}
    \end{array}\right.
	\end{equation}
    where
%    $\frac{d\nu}{dx}$ denotes the normalized equilibrium measure on the interval $[\alpha,\beta]$ and $f^{[\nu]}$ the corresponding Markov function given in \eqref{sqrt}, %with Markov function $f^{[\nu]}(z)=\frac{\sqrt{|\alpha|}}{\sqrt{(z-\alpha)(z-\beta)}}$,   and
    $$
         \eta_{2m}=\max_{z\in\mathbb E}|G_{2m}(z)|, \quad G_{2m}(z)=\prod_{j=1}^{2m} \frac{\varphi(z)-\varphi(z_{j})}{1-\varphi(z)\varphi(z_{j})}
    $$
    with $\varphi$ mapping conformally $\overline{\mathbb{C}}\backslash [\alpha,\beta]$ onto the complement of the closed unit disk.
    %These bounds are attaint up to a factor $2$ (??) for the particular Markov function $f_0$ of \eqref{}.
\end{Thm}
\begin{proof}
    Define
    \begin{equation}\label{equaomega}
             \omega(z)=\pm \prod_{j=1}^{2m}(z-z_{j}) .
    \end{equation}
    where by assumption on $z_1,...,z_{2m}$ the function $\omega$ is real-valued on the real axis and different from zero in $[\alpha,\beta]$, we hence may fix the sign such that $\omega(z)>0$ for $z\in[\alpha,\beta]$. It was probably Gonchar in \cite{Gonch78} who observed first that the denominator $Q_m$ of the interpolant $r_m^{[\mu]}$ of our Markov function $f^{[\mu]}$ is necessarily a scalar multiple of an $m$th orthonormal polynomial with respect to the measure $d\mu/\omega$, in particular the rational interpolant $r_m^{[\mu]}$ exists, is unique, and has $m$ simple poles in $(\alpha,\beta)$, with positive residuals, see also \cite[Lemma~6.1.2]{Tot92}. Gonchar also gave the integral formula %\eqref{equaerrorthog}
    \begin{equation} \label{equaerrorthog}
       \forall z \in \mathbb C\setminus [\alpha,\beta]: \quad f^{[\mu]}(z)-r_m^{[\mu]}(z)
       =
         \frac{\omega(z)}{Q_{m}^{2}(z)} \inalbe \frac{Q_{m}^{2}(x)}{\omega(x)}\frac{d\mu(x)}{z-x}        .
     \end{equation}
     In our approach we use two polynomial extremal problems: we claim that
     \footnote{In other approaches like in \cite[Section~6.1]{Tot92} the authors eliminate the term $1/(z-x)$ in the integral leading to bounds for the absolute error.}
     \begin{eqnarray}
       \forall z \in \mathbb R\setminus [\alpha,\beta]:  \quad
       |f^{[\mu]}(z)-r^{[\mu]}_m(z)| &=& \label{equaborn}
        \min_{\deg Q \leq m} \frac{|\omega(z)|}{Q_{}^{2}(z)} \inalbe \frac{Q_{}^{2}(x)}{\omega(x)}\frac{d\mu(x)}{|z-x|}
       \\&\leq& \label{minpoly}
       |f^{[\mu]}(z)| \, \min_{\deg Q \leq m} \frac{|\omega(z)|}{Q_{}^{2}(z)} \, \| \frac{Q_{}^{2}}{\omega} \|_{L^\infty([\alpha,\beta])}.
    \end{eqnarray}
    Indeed, for any $z\in \mathbb R\setminus \supp(\mu)$ it follows from \cite[Thm 3.1.3 and 3.1.4]{Sze75}
     that the denominator $Q_m$ is extremal for the extremal problem on the right-hand side of \eqref{equaborn}, and hence the claimed equality \eqref{equaborn} follows from \eqref{equaerrorthog}, whereas inequality \eqref{minpoly} is a trivial consequence of \eqref{equaborn} and of the fact that $\supp(\mu)\subset [\alpha,\beta]$ by \eqref{eq_markov}.
     It remains to solve the $L^\infty$ extremal problem on the right-hand side of \eqref{minpoly}, and again we will see that we get the same extremal polynomial for all $z\in \mathbb R \setminus [\alpha,\beta]$, namely the weighted Chebyshev polynomial. By the theory of best approximation and the Chebyshev theorem \cite[Section 4.1 and 4.4]{Mein67}, all we have to do is to find a polynomial $P$ of degree $\leq m$ such that $P(x)/\sqrt{\omega(x)}$ is of modulus $\leq 1$ on $[\alpha,\beta]$, and takes $m+1$ times in $[\alpha,\beta]$ alternately the values $1$ and $-1$.

    We first show that it is sufficient to consider the interval $[-1,1]$. Let $T\in \mathcal R_{1,1}$ be a Moebius transform with $T([-1,1])=[\alpha,\beta]$ and $T(\mathbb R)=\mathbb R$, and define $w=\varphi(z)$ by the formula $z=T(\tfrac{1}{2}(w+\tfrac{1}{w}))$, then $\varphi$ is a conformal bijection of the exterior of the interval $[\alpha,\beta]$  onto the exterior of the unit disk.\footnote{The interested reader may observe that we do not impose a normalization condition, and hence neither $T$ nor $\varphi$ are unique, though the function $G_{2m}$ can be shown to be unique, see also Remark~\ref{rem_eta}.} For any polynomial $p$ of degree $\leq m$ we find a polynomial $P$ of degree $\leq m$ such that
    \begin{equation} \label{eq_change_variables}
            \frac{P(T(y))}{\sqrt{\omega(T(y))}} = \frac{p(y)}{\sqrt{\rho(y)}} , \quad
            \mbox{where} \quad \rho(y) = \pm \prod_{j=1}^{2m}(y-T^{-1}(z_{j}))
    \end{equation}
    and, as in \eqref{equaomega}, the sign is chosen such that $\rho>0$ in $[-1,1]$.
    Especially, with $p$, also  $P$ has the desired oscillatory behavior. It remains to construct $p$, which is explained in \cite[Section 4.4]{Mein67} if $\omega$ is a square of a polynomial (the case of points of even multiplicity) but easily extends to our more general setting.
    % \footnote{Il faudra peut-etre mieux chercher si cette formule n'existe pas quelque part dans la litterature}.
    We may factorize
    $$
          \rho(\frac{1}{2}(w+\frac{1}{w})) = H(w)H(\frac{1}{w}) ,
          \quad H(w)=\sum_{k=0}^{2m} H_k w^k = H_{2m} \prod_{j=1}^{2m} (w-w_j) , \quad |w_j|>1
    $$
    where $\frac{1}{2}(w_j+\frac{1}{w_j}) = T^{-1}(z_{j})$, in other words, $w_j=\varphi(z_j)$ for $j=1,...,2m$. Since $w^\ell + w^{-\ell}$ is a polynomial of degree $\ell$ of $w+w^{-1}$ for $\ell=0,...,m$, we conclude that $p$ defined by
    $$
         p(\frac{1}{2}(w+\frac{1}{w}) = \frac{1}{2} ( w^{-m}H(w)+w^m H(\frac{1}{w}) )
    $$
    is a polynomial of degree $\leq m$. Introduce the Blaschke product
    $$
         B(w) = \frac{w^{2m} H(\frac{1}{w})}{H(w)} = \prod_{j=1}^{2m} \frac{1-w_j w}{w-w_j} = \frac{1}{G_{2m}(\varphi^{-1}(w))}
    $$
    having all its zeros in $\mathbb D$, non-real zeros occurring in conjugate pairs.
    Then for $x=\cos(t)$ and $w=e^{it}$ we have that
    $$
           \frac{w^{-m}H(w)}{|w^{-m}H(w)|} = e^{-is} , \quad
           \frac{w^{m}H(1/w)}{|w^{m}H(1/w)|} = e^{is} , \quad B(w)= e^{2is},
    $$
    and hence $p(\cos(t))/\sqrt{\rho(\cos(t))} = \cos(s)$.
    However, for a Blaschke product as above, we know that with $t\in [0,\pi]$, $2s$ runs trough the interval $[0,2m\pi]$, leading to the desired oscillatory behavior. To summarize, we have shown that
    \begin{equation} \label{Q_extremal}
        \min_{\deg Q \leq m} \| \frac{|\omega|}{Q_{}^{2}} \|_{L^\infty(\mathbb E)} \, \| \frac{Q_{}^{2}}{\omega} \|_{L^\infty([\alpha,\beta])}
        = \max_{x\in \mathbb E} \min_{\deg Q \leq m} | \frac{\omega(x)}{Q_{}^{2}(x)} | \, \| \frac{Q_{}^{2}}{\omega} \|_{L^\infty([\alpha,\beta])}
        = \max_{x\in \mathbb E} \frac{4|G_{2m}(x)|}{(1+G_{2m}(x))^2},
    \end{equation}
    with $G_{2m}(x)\in (-1,1)$ for $x\in \mathbb E$ in the general case. In the positive case we know in addition that $G_{2m}(\beta)=1>0$, $G_{2m}$ has an even number of sign changes in any subinterval of $\mathbb R\setminus Int(\mathbb E)$, and only zeros with even multiplicities in $Int(\mathbb E)$. Hence $G_{2m}(x)\in [0,1)$ for $x\in \mathbb E$ in the positive case. We still need to show that we may express rational interpolants of the Markov function $f^{[\nu]}$ in terms of $G_{2m}$.  With the same change of variables as above, $w=\varphi(z)$, $z=T(y)$, $y=(w+1/w)/2$, we claim that there exist rational functions $r,R\in \mathcal R_{m-1,m}$ such that, for $z\not\in [\alpha,\beta]$,
    \begin{equation} \label{G_value}
        \frac{1-G_{2m}(z)}{1+G_{2m}(z)} = \sqrt{y^2-1} r(y) = \frac{R(z)}{f^{[\nu]}(z)}.
    \end{equation}
    Here the first identity is obtained by taking the above $p(y)$ as denominator, and the second is left to the reader. Observing that the left-hand side of \eqref{G_value} equals one iff $z\in \{ z_1,...,z_{2m}\}$, we conclude that $r(y)$ is the rational interpolant of type $[m-1|m]$ of $1/\sqrt{y^2-1}$ at the nodes $y_j=T^{-1}(z_j)$,
    and $R(z)=r_{m}^{[\nu]}(z)$ is the rational interpolant of type $[m-1|m]$ of $f^{[\nu]}(z)$ at the nodes $z_j$. In particular, for $z\in \mathbb E$,
    \begin{eqnarray*} &&
         \frac{4 | G_{2m}(z)|}{(1+G_{2m}(z))^2}
         = | 1- \Bigl(\frac{r_{m}^{[\nu]}(z)}{f^{[\nu]}(z)}\Bigr)^2 |
         \leq \frac{4 \eta_{2m}}{(1-\eta_{2m})^2} \quad \mbox{in the general case},
         \\
         &&
         \frac{4 | G_{2m}(z)|}{(1+G_{2m}(z))^2} \leq
         \frac{4 \, | G_{2m}(z)|}{1+G_{2m}(z)} \leq
         2 \, | 1- \frac{r_{m}^{[\nu]}(z)}{f^{[\nu]}(z)} |
         \leq 4 \eta_{2m}
         \quad \mbox{in the positive case}.
    \end{eqnarray*}
    Combining with \eqref{minpoly} and \eqref{Q_extremal}, we arrive at the conclusion \eqref{equaprinc} of Theorem~\ref{thmprincR}.
\end{proof}

In a later section, it will be necessary to estimate the relative error $1-r_m^{[\mu]}/f^{[\mu]}$ at a matrix argument $A$, without computing explicitly $f^{[\mu]}(A)$. We suggest two bounds, the first one
following directly from Theorem~\ref{thmprincR} and the observation that $(f^{[\nu]}(A))^{-2}$ is easy to evaluate, leading to a kind of residual for the inverse square root. The second approach, based on a generalization of Theorem~\ref{thmprincR}, states that the relative error does not change much if one replaces $f^{[\mu]}$ by $r_{m+m'}^{[\mu]}$ for a modest value of $m'$. Such a trick was also applied as a heuristic in error estimates for matrix functions times a vector, especially for the case $f(z)=1/z$ of solving systems on linear equations, see \cite{Golub10} and the references therein. %\footnote{\textcolor{red}{G. H. Golub and G. Meurant, Matrices, Moments and Quadrature with Applications, Princeton University Press, 2010}}
\begin{Cor}\label{cor_posteriori}
    With the notations of $Theorem~\ref{thmprincR}$, let $A$ be a symmetric matrix with spectrum $\mathbb E=\sigma(A)$ out of $[\alpha,\beta]$. Then we have the {\itshape residual bound}
    $$
         \| I - r_{m}^{[\mu]}(A) \Bigl(f^{[\mu]}(A)\Bigr)^{-1} \|
         \leq \| I - r_{m}^{[\nu]}(A)^2 \frac{1}{|\alpha|}(A-\alpha I)(A-\beta I) \| .
    $$
    If in addition $\eta_{2m}\leq (\sqrt{2}-1)^2$, and
    $$
            \delta := \frac{4\widetilde \eta}{(1-\widetilde \eta)^2} \in (0,1), \quad
            \widetilde \eta := \max_{z\in \mathbb E} \Bigl| \prod_{j=2m+1}^{2m+2m'}
            \frac{\varphi(z)-\varphi(z_j)}{1-\varphi(z)\varphi(z_j)} \Bigr|,
    $$
    then we have the {\itshape a posteriori bound}
    $$
         \| I- r_{m}^{[\mu]}(A) \Bigl(f^{[\mu]}(A)\Bigr)^{-1} \| \leq
         \frac{1+\delta}{1-\delta}
         \| I- r_{m}^{[\mu]}(A) \Bigl(r_{m+m'}^{[\mu]}(A)\Bigr)^{-1} \| .
    $$
\end{Cor}
\begin{proof}
    The first claim is an immediate application of Theorem~\ref{thmprincR}.
    For the second one we vary slightly the argument in \eqref{minpoly} (with $m$ replaced by $m+m'$): instead of taking a general polynomial $Q$ of degree $\leq m+m'$, we take a polynomial $Q=PQ_m$, with $P$ of degree $\leq m'$, and obtain for $z\in \mathbb E$ that
    $$
      \Bigl| \frac{f^{[\mu]}(z)-r_{m+m'}^{[\mu]}(z)}{f^{[\mu]}(z)-r_{m}^{[\mu]}(z)} \Bigr|
    \leq \min_{\deg P \leq m'} \frac{|\widetilde\omega(z)|}{P_{}^{2}(z)} \, \| \frac{P_{}^{2}}{\widetilde\omega} \|_{L^\infty([\alpha,\beta])} , \quad
    \widetilde \omega(z)=\prod_{j=2m+1}^{2m+2m'} (z-z_j) .
    $$
    Proceeding as in the proof of Theorem~\ref{thmprincR} we conclude that
    $$
              \Bigl| \frac{1-r_{m+m'}^{[\mu]}(z)/f^{[\mu]}(z)}{1-r_{m}^{[\mu]}(z)/f^{[\mu]}(z)} \Bigr| \leq \frac{4 \widetilde \eta}{(1-\widetilde \eta)^2} = \delta \in (0,1) ,
              \quad
              | 1-r_{m}^{[\mu]}(z)/f^{[\mu]}(z) | \leq \frac{4 \eta_{2m}}{(1-\eta_{2m})^2} \leq 1,
    $$
    the last inequality following from assumption on $\eta_{2m}$. As a consequence,
    \begin{eqnarray*} &&
       \Bigl| \frac{1- r_{m}^{[\mu]}(z)/f^{[\mu]}(z)}{1- r_{m}^{[\mu]}(z)/r_{m+m'}^{[\mu]}(z)} \Bigr|
       \leq  \Bigl| \frac{r_{m+m'}^{[\mu]}(z)}{f^{[\mu]}(z)} \Bigr| \, \frac{|1- r_{m}^{[\mu]}(z)/f^{[\mu]}(z)|}{
       |1- r_{m}^{[\mu]}(z)/f^{[\mu]}(z)| - |1- r_{m+m'}^{[\mu]}(z)/f^{[\mu]}(z)|} \leq \frac{1+\delta}{1-\delta},
    \end{eqnarray*}
    as required to conclude.
\end{proof}
%We will refer to the first bound as , and the second one without the factor $\frac{1+\delta}{1-\delta}>1$ as an a posteriori bound. In our numerical experiments  we will draw  this bound for fixed $m'$ and varying $m$... \footnote{to be discussed after having drawn, to to choose the $2m'$ extra interpolation points, if possible disjoint from the first $2m$ one, $\delta$ small?}

\begin{Rem}\label{rem_eta}
    In order to make the rate of convergence in Theorem~\ref{thmprincR} more explicit (which may guide us in the choice of "good" interpolation points), we need a more precise knowledge on the quantity $\eta_{2m}=\eta_{2m}(\mathbb E)$, what is possible for intervals.
    Let $[c,d]\subset \mathbb R \setminus [\alpha,\beta]$ a closed interval containing $\mathbb E$ (and possibly containing $\infty$), such that
    $\eta_{2m}(\mathbb E)\leq \eta_{2m}([c,d])$. Since the composition of two Blaschke factors is a Blaschke factor, it turns out that the rate $\eta_{2m}$ in Theorem~\ref{thmprincR} does not depend on the particular choice of the Moebius map $T$, that is, $\varphi$ is not necessarily normalized at infinity, all we need is that $T(\mathbb R)=\mathbb R$, $T(-1)=\alpha$, $T(1)=\beta$, and $T$ is increasing in $[-1,1]$. There exists however a unique such $T$ which satisfies in addition $T(1/\kappa)=c$ and $T(-1/\kappa)=d$, where the value of $\kappa\in (0,1)$ is uniquely obtained by observing that the cross ratio of four co-linear reals is invariant under linear transformations, that is,
    \begin{equation} \label{eq_choice_T}
        T(-1)=\alpha, \, \,
        T(1)=\beta, \, \,
        T(\frac{1}{\kappa})=c, \, \,
        T(-\frac{1}{\kappa})=d, \quad
         \frac{(c-\alpha)(d-\beta)}{(c-\beta)(d-\alpha)} = \Bigl( \frac{1+\kappa}{1-\kappa} \Bigr)^2
         =:\frac{1}{k^2},
    \end{equation}
    and in addition $T([-1,1])=[\alpha,\beta]$, $T([\frac{1}{\kappa},-\frac{1}{\kappa}])=[c,d]$, where
    $[\frac{1}{\kappa},-\frac{1}{\kappa}]=\overline{\mathbb R}\setminus (-\frac{1}{\kappa},\frac{1}{\kappa})$. Using this Moebius map, we get the simplified expression
    \begin{equation} \label{eq_new_eta}
       \eta_{2m}([c,d]) = \max_{w\in [\frac{1}{\lambda},-\frac{1}{\lambda}]} | \prod_{j=1}^{2m} \frac{w-w_j}{1-w_j w}| , \quad w_j = \varphi(z_j) , \quad \lambda=\frac{1}{\varphi(c)}=-\frac{1}{\varphi(d)} = \frac{1-\sqrt{k}}{1+\sqrt{k}}.
    \end{equation}
    This also allows to find configurations of interpolation points which lead to small $\eta_{2m}([c,d])$: for instance for a Padé approximant at $z_1$ we have that $z_1=...=z_{2m}$, with the optimal choice $z_1=\varphi^{-1}(\infty)\in (c,d)$, leading to the rate $\eta_{2m}([c,d]) = \lambda^{2m}$. The same rate is obtained for the two-point Padé approximant with $z_1,...,z_{m}=c, z_{m+1},...,z_{2m}=d$, we refer to \cite{these_Bisch} for further details.
\end{Rem}

Finally, minimizing \eqref{eq_new_eta} over all choices of $w_j$ of modulus $>1$ leads to the problem of minimal Blaschke products (after the substitution $u=1/w$) on the interval $[-\lambda,\lambda]$, which has been recently reviewed in \cite{TC15}. We summarize these findings in the following corollary.
\begin{Cor}\label{cor_a_priori}
    Let $\alpha,\beta,c,d,k,\lambda,\varphi$ with $\mathbb E\subset [c,d]$ be as in Remark~\ref{rem_eta}. Then the optimal nodes minimizing  $\eta_{2m}([c,d])$ are given in terms of Jacobi elliptic functions $\mbox{sn}(\cdot,\cdot)$ \cite{Abra64} and the complete elliptic integral $K(\cdot)$ \cite{Abra64} by
    \begin{equation} \label{optimal_nodes}
       \frac{1}{\varphi(z_j)} = \frac{1}{w_j} =
        \lambda \, \mbox{sn}\Bigl(K(\lambda^2)\bigl(-1+\frac{2j-1}{2m}\bigr),\lambda^2\Bigr) \in (-\lambda,\lambda)
    \end{equation}
    for $j=1,2,...,2m$, leading to the a priori bound
	\begin{equation}\label{optimal_nodes2}
	   \|\frac{f^{[\mu]}-r_{m}^{[\mu]}}{f^{[\mu]}}\|_{L^{\infty}([c,d])} \leq 8 \rho^{2m}/(1-2\rho^{2m})^2 ,
       \quad \rho:= \exp( \frac{-1}{\capa([\alpha,\beta],[c,d])} ),
	\end{equation}
    provided that $2\rho^{2m}<1$.
\end{Cor}
\begin{proof}
    In \cite[Problem D, p. 112]{TC15}, the authors recall the following link with the third Zolotarev problem
    $$
         \eta_{2m}([c,d]) = \min_{\mbox{\footnotesize $B$ Blaschke of order $2m$}} \, \| B \|_{L^\infty([-\lambda,\lambda])}
         = \sqrt{\min_{R\in \mathcal R_{2m,2m}} \| R \|_{L^\infty([-\lambda,\lambda])} \, \| \frac{1}{R} \|_{L^\infty([1/\lambda,-1/\lambda]) }}
          ,
    $$
    and give explicitly in \cite[Section 3.2, p.109]{TC15} the roots \eqref{optimal_nodes} of an optimal Blaschke product.
    %\footnote{Joanna, je compte sur vous pour completer, je pense que l'on passe a un autre intervalle, non?...}
    Here \cite[Corollary 3.2]{BBTS19} gives us asymptotically sharp upper bound for this Zolotarev number and thus $\eta_{2m}([c,d])$ for optimal $z_j$ as in \eqref{optimal_nodes}, namely
    \begin{equation} \label{optimal_rate}
       \eta_{2m}([c,d]) \leq 2 \, \exp( - \frac{m}{\capa([-\lambda,\lambda],[1/\lambda,-1/\lambda])} )
       = 2 \, \rho^{2m}, %\quad \rho:= \exp( \frac{-1}{\capa([\alpha,\beta],[c,d])} ),
    \end{equation}
    the last equality following by symmetry and by the fact that the logarithmic capacity is invariant under conformal mappings of the underlying doubly connected domain. Combining with Theorem~\ref{thmprincR} and using $\eta_{2m}(\mathbb E) \leq \eta_{2m}([c,d]))$, we get the claimed inequality \eqref{optimal_nodes2}.
\end{proof}

\begin{Rem}\label{rem_stopping}
    With the notations of Theorem~\ref{thmprincR}, we may even slightly improve the statement of Corollary~\ref{cor_a_priori}: combining Theorem~\ref{thmprincR} and \eqref{optimal_rate} we find that the relative error for $r_m^{[\mu]}$ is bounded above by the residual error for $r_m^{[\nu]}$, which itself is bounded above by the a priori bound given in \eqref{optimal_nodes2}.
    Given a symmetric matrix $A$ with spectrum $\sigma(A)\subset [c,d]$, we will report in Figure~\ref{fig_representation1} and \S\ref{sec_experiences} about numerical experiments showing that, due to finite precision, the computed rational interpolants, evaluated on a computer at a matrix argument $A$, do no longer respect these inequalities, More precisely, the relative error for $r^{[\mu]}_m(A)$ has a quite erratic behavior once the error is no longer smaller than the a priori bound. In order to find the index $m$ corresponding to smallest error, we suggest to compute $r_m^{[\mu]}(A)$ and $r_m^{[\nu]}(A)$ for $m=1,2,..$ and stop one index before the first $m$ where the residual error is larger than five times the a priori bound, that is, where
    \begin{equation} \label{eq_stopping}
         \| I - r_{m}^{[\nu]}(A) \frac{1}{|\alpha|}(A-\alpha I)(A-\beta I) r_{m}^{[\nu]}(A) \|  \geq
         40 \rho^{2m}/(1-2\rho^{2m})^2.
    \end{equation}
    For the relative error curves presented in Figure~\ref{fig_representation1} we have displayed all indices $m$ verifying \eqref{eq_stopping} by markers. It turns out that this heuristic stopping criterion, which implicitly assumes that the floating point errors for $r_m^{[\mu]}(A)$ and $r_m^{[\nu]}(A)$ are about the same, does not require the a priori knowledge of $f^{[\mu]}(A)$, and seems to work very well in practice.
\end{Rem}

    Notice that the interpolation points \eqref{optimal_nodes} are just optimal for our upper bound, but not necessarily for the relative interpolation error. Indeed, one expects sharper bounds to hold if the support of $\supp(\mu)$  consisting for instance of two intervals is a proper subset of $[\alpha,\beta]$, or $\mathbb E$ is a proper subset of $[c,d]$. However, if $\supp(\mu)=[\alpha,\beta]$ and $\mu$ is regular in the sense of \cite{Tot92}, then Gonchar
    \cite[Theorem 1]{Gonch78bis} (see also \cite[Theorem~6.2.2]{Tot92}) showed that the $2m$th root of the error of best approximation of $f$ on $[c,d]$ in
    $\mathbb \mathcal R_{m-1,m}$ tends to $\rho$ for $m\to \infty$. More precisely, for the particular Markov function $f^{[\mu]}(z)=1/\sqrt{z}$ with $[\alpha,\beta]=[-\infty,0]$, inequalities for the relative error of best approximation of $f^{[\mu]}$ in
    $\mathbb \mathcal R_{m-1,m}$ are known since the work of Zolotarev\cite{Zol77}, who expressed several extremal problems in terms of Zolotarev numbers, see also \cite[p. 147]{Ach90} and the Appendix of \cite{BBTS19}, from which it follows that the relative error is $\leq 4 \rho^{2m}$, and behaves like $4\rho^{2m} (1+o(\rho^{2m})_{m\to \infty})$, see also, e.g., \cite[Theorem~V.5.5]{bra86}.
    Hence, if we want interpolation points which work for any Markov function, the interpolation points \eqref{optimal_nodes} are optimal up to a factor at most $2$.

%Include 2 graphics ??? for comparing best Padé, two-point Padé and "best" interpolant, e.g. for \eqref{sqrt}, one in the case of fast convergence, once in the case where $\beta\approx c$

\subsection{The disk case}\label{sec2.2}

%\footnote{Il y avait un probleme avec ce chapitre, vous utilisez les mêmes $c$ et $d$ dans vos applications conformes, non? Vous n'avez pas demontre le corollaire 2.3, que me semble demander que $z_j\in [1/\alpha,0]$. De plus, on peut eliminer le lemme de Freud en obtenant des estimations plus fines, ...}

With interpolation points $z_1,...,z_{2m}$ as before, we still keep the integral representation \eqref{equaerrorthog} for the error for complex $z$, but \eqref{equaborn}
and \eqref{minpoly} are no longer true, and need to be adapted following the technique of the Freud Lemma \cite[section III.7]{Freud71}. This result has also been exploited in the work of Ganelius \cite{Gan82} and of Braess \cite[Theorem~2.1]{Braess87}. The work of these authors on interpolation points with even multiplicity did inspire us in the proof of the following theorem, but we had to correct an erroneous application of the Freud Lemma in \cite[Eqn. (2.7)]{Braess87} where an additional $\beta-\alpha$ factor should occur, and could improve \cite[Eqn. (2.7)]{Braess87} by a factor $2$. We also give an estimate for interpolation points of arbitrary multiplicity which to our knowledge is new.

\begin{Thm}\label{thmprinc2}
    Let $-\infty\leq \alpha<\beta<-1$, and let the Markov function $f^{[\mu]}$ be as in \eqref{eq_markov}.
    Consider interpolation points $z_1,...,z_{2m}\in\mathbb C\setminus [\alpha,\beta]$ where we suppose that non-real points only occur in conjugate pairs.
    If all interpolation points have even multiplicity, then\footnote{Notice that for the critical case $\beta$ close to $-1$ and $\mu$ a probability measure, the above constant $C$ behaves at worst as $1/\mbox{dist}(\overline{\mathbb D},[\alpha,\beta])=1/(-1-\beta)$, a term which also occurs in the works of Ganelius \cite{Gan82} and Braess \cite{Braess87}.}
    $$
          \| f^{[\mu]} - r_m^{[\mu]} \|_{L^\infty(\overline{\mathbb D})}
          \leq C \, \max_{z\in[\alpha,\beta]} \Bigl| \prod_{j=1}^{2m} \frac{1-zz_j}{z-z_j} \Bigr| ,
          \quad C = \frac{1-\beta}{-1-\beta}\, f(-1).
    $$
    In the general case, $$
          \| f^{[\mu]} - r_m^{[\mu]} \|_{L^\infty(\overline{\mathbb D})}
          \leq C \, \frac{4 \eta'_{2m}}{(1-\eta'_{2m})^2} , \quad \eta'_{2m} = \max_{z\in \overline{\mathbb D}}
          G_{2m}(z)
	$$
    where $C$ is as before, and $G_{2m}$ as in Theorem~\ref{thmprincR}.
\end{Thm}
\begin{proof}
   Define $\omega$ as in \eqref{equaomega}.
   Let $Q$ be any polynomial of degree at most $m$ with real coefficients, then $x\mapsto \frac{Q(x)/Q(z)-1}{z-x}$ is a polynomial of degree at most $m-1$, and thus orthogonal to $Q_m$ with respect to the measure $\mu/\omega$. This leads to the well-known fact that
   $$
           f^{[\mu]}(z)-r^{[\mu]}_m(z) =
         \frac{\omega(z)}{Q_{m}^{2}(z)} \inalbe \frac{Q_{m}^{2}(x)}{\omega(x)}\frac{d\mu(x)}{z-x}
         = \frac{\omega(z)}{Q_{m}(z)} \inalbe \frac{Q_{m}(x)}{\omega(x)}\frac{d\mu(x)}{z-x}
         = \frac{\omega(z)}{Q_{m}(z)Q(z)} \inalbe \frac{Q_{m}(x)Q(x)}{\omega(x)}\frac{d\mu(x)}{z-x}.
   $$
   We now observe that
   \begin{equation} \label{inequality_disk}
     \mbox{for $|z|=1$ and $x\leq \beta<-1$:} \quad
        \frac{1}{1-x} \leq \mbox{Re}(\frac{1}{z-x}) \leq \frac{1}{|z-x|} \leq \frac{1}{-1-x}
        \leq  \frac{1}{1-x} \frac{1-\beta}{-1-\beta}   ,
   \end{equation}
   and apply the Cauchy-Schwarz inequality in the last integral in order to obtain
   \begin{eqnarray*} |f^{[\mu]}(z)-r_m^{[\mu]}(z)|^2
       &\leq&
       \Bigl| \frac{\omega(z)}{Q_{m}^{2}(z)} \Bigr| \, \inalbe \frac{Q_{m}^{2}(x)}{\omega(x)}\frac{d\mu(x)}{|z-x|} \,
       \Bigl| \frac{\omega(z)}{Q_{}^{2}(z)} \Bigr|\,
       \inalbe  \frac{Q_{}^{2}(x)}{\omega(x)}\frac{d\mu(x)}{|z-x|}  .
       \\
       &\leq&
       \frac{1-\beta}{-1-\beta}\, |f^{[\mu]}(z)-r^{[\mu]}_m(z)| \,
       \Bigl| \frac{\omega(z)}{Q_{}^{2}(z)} \Bigr|\,
       \inalbe  \frac{Q_{}^{2}(x)}{\omega(x)}\frac{d\mu(x)}{|z-x|}  ,
   \end{eqnarray*}
   the second inequality following from \eqref{inequality_disk}. Thus we get for  $|z|=1$ the following upper bounds for the absolute and relative interpolation errors
   \begin{eqnarray} \label{goal_disk}
      && |f^{[\mu]}(z)-r^{[\mu]}_m(z)| \leq C \,
      \min_{\deg Q \leq m} \| \frac{\omega}{Q_{}^{2}}\|_{L^\infty(\partial \mathbb D)} \, \| \frac{Q_{}^{2}}{\omega} \|_{L^\infty([\alpha,\beta])},
      \\ \nonumber
      && \Bigl|1 - \frac{r^{[\mu]}_m(z)}{f^{[\mu]}(z)} \Bigr| \leq (\frac{1-\beta}{-1-\beta})^2 \,
      \min_{\deg Q \leq m} \| \frac{\omega}{Q_{}^{2}}\|_{L^\infty(\partial \mathbb D)} \, \| \frac{Q_{}^{2}}{\omega} \|_{L^\infty([\alpha,\beta])},
   \end{eqnarray}
   the second following from the first since $C=\frac{1-\beta}{-1-\beta}\, f^{[\mu]}(-1) \leq (\frac{1-\beta}{-1-\beta})^2 \, \mbox{Re}f^{[\mu]}(z)\leq (\frac{1-\beta}{-1-\beta})^2 \, |f^{[\mu]}(z)|$, again by \eqref{inequality_disk}.

   In the case of interpolation points of even multiplicity, say, $z_{m+j}=z_j$ for $j=1,...,m$, we get the upper bound claimed in Theorem~\ref{thmprinc2} by taking $Q(x)=(1-z_1z)...(1-z_mz)$, which can be shown using the maximum principle for analytic functions to be the extremal polynomial in \eqref{goal_disk} in this special case. Finally, in the general case we use the same polynomial for the interval $[\alpha,\beta]$ as in the proof of Theorem~\ref{thmprincR}, which can be shown to be optimal for \eqref{goal_disk} up to the factor $4/(1-\eta'_{2m})^2$.
\end{proof}

\begin{Rem}
   As in the previous chapter, one may ask for a single optimal interpolation point $z_1=...=z_{2m}$, or for a configuration of distinct points minimizing $\eta'_{2m}$. Here the results of the previous chapter remain valid, by choosing $[c,d]=[1/\beta,1/\alpha]$ in \eqref{eq_choice_T}, we refer the reader to  \cite{these_Bisch} for further details. As a rule of thumb, "good" interpolation points are in $[1/\beta,1/\alpha]$.

   In contrast, in his study of Faber-Padé approximants \cite{Kniz08}, Knizhnerman considered the special case $z_1=...=z_{2m}=0$ and hence $\eta_{2m}'=(1/\beta)^{2m}$. In this or in the more general case $z_j \in [1/\alpha,0]$, the error analysis simplifies considerably: since the maximum of the error $|f^{[\mu]}-r^{[\mu]}_m|$ on the unit circle can be shown to be attained at $z=-1$, and the same is true for the statement of Theorem~\ref{thmprinc2}. In addition, the factor $\frac{1-\beta}{-1-\beta}$ can be dropped.
\end{Rem}

\begin{Rem}\label{rem_FoV}
   Braess \cite{Braess87} used the Carath\'eodory-Fej\'er method to derive $L^\infty([-1,1])$ estimates from $L^\infty(\overline{\mathbb D})$ estimates for the interpolation error. Comparing our Theorems~\ref{thmprincR} and~\ref{thmprinc2}, our interval estimates seem to be sharper.

   More generally, for a general convex compact set $\mathbb E$ being symmetric with respect to the real axis, one may use the Faber map $\mathcal F$ (see \cite{Gaier87} or \cite{beckermann09}) and its modification $\mathcal F_+(h)=\mathcal F(h)+h(0)$  to get
   good rational approximants on $\mathbb E$ from those on $\mathbb D$. More precisely,
   Ellacott \cite[Theorem 1.1]{Ella83} showed that $r\in \mathcal R_{m-1,m}$ iff $\mathcal F(r)\in \mathcal R_{m-1,m}$ and simultaneously \cite{Kniz08} and \cite{beckermann09} found out that the Faber pre-image of a Markov function is a Markov function, with explicit formulas for the measure. Thus one may use the inequality
   $$
        \| \mathcal F_+(f^{[\mu]}) - \mathcal F_+(r_m^{[\mu]}) \|_{L^\infty(\mathbb E)} \leq  2 \,
        \| f^{[\mu]} - r_m^{[\mu]} \|_{L^\infty(\overline{\mathbb D})}
   $$
   shown in \cite[Theorem 2]{Gaier87}
   together with our findings in Theorem~\ref{thmprinc2} to find good rational approximants on $\mathbb E$ for Markov functions. In the context of Markov functions of matrices, we should also mention the result \cite[Theorem~2.1]{beckermann09} that
   \begin{equation} \label{est_function_of_matrices}
        \| \mathcal F_+(f^{[\mu]})(A) - \mathcal F_+(r^{[\mu]}_m)(A) \| \leq  2 \,
        \| f^{[\mu]} - r^{[\mu]}_m \|_{L^\infty(\overline{\mathbb D})}
   \end{equation}
   provided that the field of values of the square matrix $A$ is a subset of $\mathbb E$.
\end{Rem}

    To summarize, in \S\ref{sec2} we presented a detailed study on the relative error (in exact arithmetic) obtained by approaching $f^{[\mu]}$ by $r_{m}^{[\mu]}$ with $r_m^{[\mu]}$ a rational interpolant of $f^{[\mu]}$ of type $[m-1|m]$ at quite arbitrary interpolation points. This allowed us to find quasi-optimal interpolation points which minimize our upper bound for the error, together with the explicit and simple a priori bound \eqref{optimal_nodes2}. Also, several new a posteriori error bounds are provided, which allowed to derive in Remark~\ref{rem_stopping} a new heuristic stopping criterion for finding the degree $m$ leading to a small relative error even in the context of floating point arithmetic. However, for a successful implementation we need to discuss how to represent and compute our interpolants.

\section{The computation of rational interpolants for distinct real interpolation nodes}\label{sec3}

The computation of a rational interpolant $r_m=P_m/Q_m$ of type $[m-1|m]$ or $[m|m]$ of $f$ or its evaluation at some argument $z\in [c,d]$ is strongly connected to the way how we represent our rational interpolant. We will suppose in what follows our (finite or infinite sequence of) interpolation points $z_j$ for $j\geq 1$ are distinct, real, and ordered such that
\begin{equation} \label{ordering}
       \beta < c \leq z_1 < z_2 < ... \leq d.
\end{equation}
\subsection{Computing separately numerator and denominator}
A first perhaps naive approach would be to represent both numerator and denominator in the same polynomial basis, and then solve for the coefficients in this basis by writing a homogeneous system of $2m$ equations and $2m+1$ unknowns translating the interpolation conditions $f(z_j)Q_m(z_j)-P_m(z_j)=0$ for $j=1,...,2m$.
Using the basis of monomials, interesting complexity results for evaluating $P_m(A)$ and $Q_m(A)$ are given in \cite{Fasi19}. %\footnote{\textcolor{red}{Massimiliano Fasi, Optimality of the Paterson–Stockmeyer method for evaluating matrix polynomials and rational matrix functions, Linear Algebra and its Applications Volume 574, 1 August 2019, Pages 182-200}}
However, in practice the underlying matrix of coefficients turns out to be quite often very ill-conditioned, there is a rule of thumb for Pad\'e approximants using monomials \cite[Section~2.1]{BGM96}
that we might loose at least $m$ decimal digits of precision in solving such systems.
Of course, other (scaled) polynomial bases, like Chebyshev polynomials scaled to the interval $[c,d]$ or a Newton basis corresponding to a suitable ordering of the interpolation points might lead to better conditioning, but a "good" basis should not only depend on $[c,d]$ but also on the function $f$ to be interpolated. Thus we have not implemented such an approach.
\subsection{Computing poles and residuals in a partial fraction decomposition}\label{sec3.1}

Since for Markov functions $f$ the interpolant $r_m$ has $m$ simple poles $x_1,...,x_m$ in $(\alpha,\beta)$, we may look directly for the partial fraction decomposition
\begin{equation} \label{pfd}
    r_m(z) = \frac{a_1}{z-x_1}+...+\frac{a_m}{z-x_m} .
\end{equation}
By the work of Mayo and Antoulas \cite{MAYO07} nicely summarized in the recent paper \cite[Section 2]{Emb19}, $r_m$ may be represented as a transfer function of a SISO dynamical system with help of a matrix pencil: we have that $r_m(z)=W(\mathbb{L}_{s}-z\mathbb{L})^{-1}V^T$ with the row vectors
$W=(f(z_{2j-1}))_{j=1,\ldots,m}$, $ V=(f(z_{2j}))_{j=1,\ldots,m} $ and the Loewner matrices
$$
     \mathbb{L}=\Bigg(
        \frac{f(z_{2j})-f(z_{2k-1})}{z_{2j}-z_{2k-1}}
      \Bigg)_{\substack{j=1,\ldots,m \\ k=1,\ldots,m}}\hspace*{0.2cm} \text{ and }\hspace*{0.2cm}
     \mathbb{L}_{s}=\Bigg(
     \frac{z_{2j}f(z_{2j})-z_{2k-1}f(z_{2k-1})}{z_{2j}-z_{2k-1}}
      \Bigg)_{\substack{j=1,\ldots,m \\ k=1,\ldots,m}}.
$$
Thus the poles are the eigenvalues of the Loewner matrix pencil $\mathbb{L}_{s}-z\mathbb{L}$ and can be computed with standard software, we used the Matlab function {\tt eig}. We then compute the residuals by a least square fitting,\footnote{However, the underlying rectangular Cauchy matrix $(\frac{1}{z_j-x_k})$ might be quite ill-conditioned, even after row or column scaling  \cite[Cor 4.2]{BBTS19}. A closer analysis seems to show that the given bound for the inverse condition number is related to the rate of best rational approximants of our $f$ on the interval $[c,d]$.}
 see Algorithm~\ref{algo_pfd}.
\begin{algorithm}[h!]
\SetAlgoLined
\KwResult{Poles $x_1,...,x_m$ and residuals $a_1,...,a_m$ in \eqref{pfd}.}
 \Begin{Find the eigenvalues $x_1,...,x_m$ of the Loewner matrix pencil $\mathbb{L}_{s}-z\mathbb{L}$ of Mayo and Antoulas\;
 Compute the solution $y=(a_1,...,a_m)^T$ of the least square problem of minimizing
 $$
    \| \Bigl(\frac{1}{z_j-x_k}\Bigr)_{j=1,...,2m,k=1,...,m} y - \Bigl(f(z_j)\Bigr)_{j=1,...,2m} \|
 $$.}
 \caption{Given a function $f$ and interpolation points $z_1,...,z_{2m}$, compute poles $z_j$ and residuals $a_j$ of the partial fraction decomposition \eqref{pfd} of the rational interpolant of $f$ of type $[m-1|m]$.}\label{algo_pfd}
\end{algorithm}

\subsection{Barycentric rational functions}\label{sec3.2}
The barycentric representation \cite{Berrut05}
%\footnote{ Ajouter reference J.-P.  Berrut,  R.  Baltensperger,  and  H.  D.  Mittelmann.   Recent  developments  in  barycentricrational interpolation.  In: Trends and Applications in Constructive Approximation, pages27–51. Springer, 2005.}
of a rational function $r\in \mathcal{R}_{m,m}$ with distinct support points $t_0,...,t_m$ is given by
\begin{equation}\label{bary}
r(z)=\sum_{j=0}^{m}\frac{\alpha_{j}}{z-t_{j}}\Bigg{/}\sum_{j=0}^{m}\frac{\beta_{j}}{z-t_{j}}
  ,
\end{equation}
we refer the reader to \cite[p.551]{Higham04} and \cite[Proposition 2.4.3]{Celis08} and the discussion in \cite[Section~2.3]{BB18} for backward and forward stability results on evaluating such rational functions. For constructing rational interpolants of type $[m|m]$ of $f$, one typically choses $\alpha_j=f(t_j)\beta_j$ ensuring that $r$ interpolates $f$ at these support points. Such interpolating rational functions with well-chosen support points have been the building block for a new implementation {\em minimax} of the rational Remez algorithm \cite{BB18}, which allows to compute best rational approximants of type $[m'|m]$ of Markov functions for $m',m\leq 40$ to machine precision in double precision arithmetic where previous implementations required high precision arithmetic to achieve this goal.

For computing the rational interpolant of type $[m|m]$, we choose $t_{j}=z_{2j+1}$ for $j=0,...,m$, and it remains to solve a homogeneous linear system for $\beta_0,...,\beta_m$ in order to get also interpolation at the points $z_{2k}$ for $k=1,...,m$, see Algorithm~\ref{algo_bary}. Notice that the underlying matrix of coefficients is the transposed of the Loewner matrix $\mathbb L$ seen in \S\ref{sec3.1}, bordered with one additional column.
\begin{algorithm}[h!]
\SetAlgoLined
\KwResult{For $j=0,1,...,m$: Support points $t_j$ and weights $\beta_j$, $\alpha_j=f(t_j) \beta_j$ in \eqref{bary}.}
\Begin{Define support points $t_j=z_{2j+1}$ for $j=0,1,...,m$\;
 Compute a solution $y=(\beta_0,...,\beta_m)^T$ of the homogeneous system of linear equations
 $$
    \Bigl(\frac{f(z_{2k})-f(t_j)}{z_{2k}-t_j}\Bigr)_{k=1,2,...,m,j=0,1,...,m} y = 0 .
 $$}
 \caption{Given a function $f$ and interpolation points $z_1,...,z_{2m+1}$, compute support points $t_j$ and weights $\beta_j$, $\alpha_j=f(t_j)\beta_j$ of the rational interpolant \eqref{bary} of $f$ of type $[m|m]$.}\label{algo_bary}
\end{algorithm}
For ensuring stability, numerical experiments show that it is mandatory that the support points and the other interpolation points interlace, see \eqref{ordering}. This is supported in \cite[Cor 4.5]{BB18} saying that, after suitable explicit column and row scaling, the Cauchy matrix $(\frac{1}{z_{2j}-z_{2k-1}})_{j,k}$ is unitary provided that we have the interlacing \eqref{ordering}.

Things become slightly more technical for the interpolant of type $[m-1|m]$, here we have taken the support points $t_0=z_1$ and $t_j=z_{2j}$ for $j=1,...,m$ ensuring nearly interlacing with the remaining interpolation points. Again we have to solve a homogeneous linear system for $\beta_0,...,\beta_m$ to impose interpolation at $z_{2k-1}$ for $k=2,...,m$, where the additional
%\footnote{Erreur (de frappe?) dans votre manuscrit.}
equation $f(t_0)\beta_0+....+f(t_m) \beta_m=0$ ensures that the degrees are correct.

\subsection{Thiele continued fractions}\label{sec3.3}
We finally turn to a representation of rational interpolants through continued fractions.
%which turns out to be our method of choice both for scalar and matrix arguments. We start with some definitions.
Following \cite[Section 7.1]{BGM96}, for given interpolation points $z_1,z_2,...$ and parameters $f^{(1)}_1,f^{(2)}_2,... \in \mathbb C$,
the $M$th convergent of a Thiele continued fraction is the rational function
\begin{equation} \label{Thiele}
      R_{M}^{(1)}(z) = f^{(1)}_1 + \cfr{z-z_1}{f^{(2)}_2}+...+ \cfr{z-z_{M-1}}{f^{(M)}_M} .
\end{equation}
We refer to a positive Thiele fraction if all $f^{(j)}_j$ are strictly positive, and \eqref{ordering} holds.
Given a function $f^{(1)}$, we define its reciprocal differences by
\begin{equation} \label{T_diff}
       \forall 1 \leq k \leq M : \quad f^{(1)}_k=f^{(1)}(z_k) , \quad
       \forall 1 \leq j < k \leq M :
       \quad \quad f^{(j+1)}_k= \frac{z_k-z_j}{f^{(j)}_k-f^{(j)}_j} ,
\end{equation}
where we tacitly suppose that there is no breakdown (that is, no division by $0$). Then $R_M^{(1)}(z_k)=f^{(1)}(z_k)$ for $k=1,...,M$, more precisely, $R^{(1)}_{2m+1}$ is the rational interpolant of type $[m|m]$ of $f^{(1)}$ at the interpolation points $z_1,...,z_{2m+1}$, and $R^{(1)}_{2m}$ is the rational interpolant of type $[m|m-1]$ of $f^{(1)}$  at the interpolation points $z_1,...,z_{2m}$. Setting $f^{(1)}(z)=1/f(z)$, we conclude that $1/R_{2m}^{(1)}$ is the desired rational interpolant of type $[m-1|m]$ of $f$.
This interpolation property becomes immediate by introducing the families of functions
\begin{equation} \label{T_functions}
    \forall 1 \leq j < k \leq M :
    \quad f^{(j+1)}(z)=\frac{z-z_j}{f^{(j)}(z)-f^{(j)}(z_j)},
    \quad R^{(j+1)}_M(z)=\frac{z-z_j}{R^{(j)}_M(z)-R^{(j)}_M(z_j)},
\end{equation}
since then $f^{(j)}_k=f^{(j)}(z_k)=R^{(j)}_M(z_k)$ for $1 \leq j \leq k \leq M$, and
$$
      R_{M}^{(j)}(z) = f^{(j)}_j + \cfr{z-z_j}{f^{(j+1)}_{j+1}}+...+ \cfr{z-z_{M-1}}{f^{(M)}_M} .
$$
In particular, $R_{M}^{(j)}$ is a rational interpolant of $f^{(j)}$ at the interpolation points $z_j,z_{j+1},...,z_M$.
The backward evaluation scheme at a fixed argument $z$ of a Thiele continued fraction given the parameters $f_j^{(j)}$ is given by
\begin{equation} \label{T_evaluation}
    R_M^{(M)}(z)=f_M^{(M)} , \quad \mbox{and for~~} j=M-1,M-2,..., 1 : \quad R_M^{(j)}(z) = f_{j}^{(j)} +
    \frac{z-z_j}{R^{(j+1)}_M(z)} .
\end{equation}
In his stability analysis of this scheme, Graves-Morris \cite{GM81} observed that, before computing $f^{(j+1)}_k$ for $k=j+1,...,M$ via \eqref{T_diff}, it is important to reorder the couples $(f^{(j)}_k,z_k)$ for $k=j,j+1,...,M$ such that, after reordering,
\begin{equation} \label{T_pivot}
       |f^{(j)}_j | = \min \{ | f^{(j)}_k| : k=j,j+1,...,M \},
\end{equation}
reminding of partial pivoting in Gaussian elimination. A combination of \eqref{T_diff}, \eqref{T_pivot}, and \eqref{T_evaluation} gives the modified Thacher-Tukey algorithm of \cite{GM81} and
\cite[Section 7.1]{BGM96}, which we have simplified a bit by omitting the case of breakdown
in \eqref{T_diff}, see Algorithm~\ref{algo_Thiele}.
\begin{algorithm}[h!]
\SetAlgoLined
\KwResult{Coefficients $f_{1}^{(1)},...,f_{M}^{(M)}$ in \eqref{Thiele}
         and value $R_M^{(1)}(z)$ of the interpolant.}
 \Begin{
   \For{$k=1,2,...,M$}{initialize $f^{(1)}_k=f^{(1)}(z_k)$\;}
   \For{$j=1,2,...,M-1$}{
      Permute $(f^{(j)}_k,z_k)$ for $k=j,j+1,...,M$ such that, after reordering,
        \eqref{T_pivot} holds\;
         \For{$k=j+1,j+2,...,M$}{$f^{(j+1)}_k=(z_k-z_j)/(f^{(j)}_k-f^{(j)}_j)$;}
    }
    initialize $R_M^{(M)}(z)=f_M^{(M)}$\;
    \For{$j=M-1,M-2,..., 1$}{$R_M^{(j)}(z) = f_{j}^{(j)} +
    \frac{z-z_j}{R^{(j+1)}_M(z)}$\;}
 }
 \caption{Given a function $f^{(1)}$ and interpolation points $z_1,...,z_{M}$, compute and evaluate at $z\in \mathbb C$ via the modified Thacher-Tukey algorithm of \cite{GM81} the Thiele continued fraction representation \eqref{Thiele} of the rational interpolant of $f^{(1)}$ of type $[m|m-1]$ (if $M=2m$) or of type $[m|m]$ (if $M=2m+1$).}\label{algo_Thiele}
\end{algorithm}

Notice that if $R_M^{(1)}$ is a positive continued fraction then by recurrence on $k-j$ using \eqref{ordering} and \eqref{T_diff} one shows that $0 < f^{(j)}_j < f^{(j)}_k$ for $1 \leq j < k\leq M$, that is, there is no breakdown in \eqref{T_diff}, and we obtain \eqref{T_pivot} without pivoting. However, we are not aware of results in the literature on classes of functions where the interpolating Thiele continued fraction is positive. Such a class is given in our first main result, the proof is presented later.

\begin{Thm}\label{Thm_Markov}
   If this is true for $1/f^{(1)}$, then all functions $1/f^{(j)}$ defined in \eqref{T_functions} are Markov functions with a measure $\mu^{(j)}$ having an infinite support $\subset [\alpha,\beta]$.
\end{Thm}

Since a Markov function as in Theorem~\ref{Thm_Markov} is positive and decreasing in $(\beta,+\infty)$, we conclude with \eqref{ordering} that $f^{(j)}(z_k) > f^{(j)}(z_j) > 0$ for $k>j$, that is, the interpolating Thiele continued fraction of $f^{(1)}$ is positive.
\begin{Ex}
   Take $f^{(1)}(z)=\sqrt{z}$ such that $1/f^{(1)}(z)$ is a Markov function with support $[\alpha,\beta]=(-\infty,0]$, a limiting case of \eqref{sqrt}. Then the reader easily verifies by recurrence that $f^{(1)}_k=\sqrt{z_k}$ and, for $j \geq 2$,
   $$
        f^{(j)}(z)=\sqrt{z} + \sqrt{z_{j-1}}, \quad f_k^{(j)} = \sqrt{z_k} + \sqrt{z_{j-1}} > 0.
   $$
   In particular, also $1/f^{(j)}(z)$ is a Markov function with support $[\alpha,\beta]=(-\infty,0]$, and the interpolating Thiele continued fraction
   $$
                \sqrt{z}=\sqrt{z_1}
                + \cfr{z-z_1}{\sqrt{z_2}-\sqrt{z_1}} + \cfr{z-z_2}{\sqrt{z_3}-\sqrt{z_2}} +...
   $$
   is positive. We have not seen before such an explicit formula for the interpolating Thiele continued fraction, only the limiting case of Padé approximants, see, e.g., \cite[Theorem~5.9]{higham2008}.
\end{Ex}

We now state and prove our second main result of this subsection on the backward stability of the modified Thacher-Tukey algorithm: an error in finite precision in \eqref{T_diff}  gives parameters of a continued fraction with exact values at $z_k$ not far from the desired values $f^{(1)}(z_k)$, provided that we use the standard model \cite[Eqn.\ (2.4)]{High02} for finite precision arithmetic between real machine numbers. In our proof of this result we have been inspired by a similar result \cite[Theorem 4.1]{GM80} of Graves-Morris, who considered non necessarily positive Thiele continued fractions with pivoting \eqref{T_pivot}, made a first order error analysis and got an additional growth factor $2^k$ for the error which we are able to eliminate.

\begin{Thm}\label{Thm_finite_precision}
   Let $1/f^{(1)}$ be a Markov function as before, and
   suppose that the quantities $\widetilde f^{(j)}_k$ for $1\leq j \leq k \leq M$ are computed via \eqref{T_diff} using finite precision arithmetic with machine precision $\varepsilon$. Denote by $\widetilde R^{(1)}_{M}$ the (exact) continued fraction constructed with the (inexact) parameters $\widetilde f_1^{(1)},...,\widetilde f_M^{(M)}$ which are supposed to be $>0$ (despite finite precision, see Remark~\ref{Rem_finite_precision}).  Then %we have the following two backward stability result
   $$
          k=1,...,M : \quad | \widetilde R^{(1)}_{M}(z_k)- f^{(1)}(z_k) |   \leq \frac{3k \varepsilon}{1-3 k^2 \varepsilon} .| \widetilde R^{(1)}_{M}(z_k) |.
   $$
\end{Thm}
\begin{proof}
   The standard model for finite precision arithmetic of \cite[Eqn.(2.4)]{High02}
   %\footnote{Cite the book of N. Higham on finite precision, in particular his lemma }
   gives the following finite precision counterpart of \eqref{T_diff}: for $k>j$
   $$
        f^{(1)}(z_k) = \widetilde f_k^{(1)}  (1+\epsilon_{1,k}), \quad
        \widetilde f_k^{(j+1)}= \frac{z_k-z_j}{\widetilde f_k^{(j)}-\widetilde f_j^{(j)}} (1+\epsilon_{j+1,k}),
   $$
   where $\epsilon_{1,k}$ comes from rounding $f^{(1)}(z_k)$, the term $\epsilon_{j+1,k}$ translates errors in the two subtractions and the division, and $|\epsilon_{j,k}|\leq \frac{3\varepsilon}{1-3\varepsilon}$ by \cite[Lemma 3.1]{High02}.
   In accordance to \eqref{T_evaluation}, we consider the rational functions defined by
   $$
    \widetilde R_M^{(M)}(z)=\widetilde f_M^{(M)} , \quad \mbox{and for~~} j=M-1,M-2,..., 1 : \quad \widetilde R_M^{(j)}(z) = \widetilde f_{j}^{(j)} +
    \frac{z-z_j}{\widetilde R^{(j+1)}_M(z)} ,
   $$
   and claim that
   \begin{equation} \label{T_claim}
          \widetilde f^{(j)}_k = (1+\delta_{j,k}) \widetilde R_M^{(j)}(z_k)  , \quad |\delta_{j,k}| \leq \gamma_{k-j} , \quad \gamma_\ell = \frac{3\ell \varepsilon}{1-3\ell^2 \varepsilon}.
   \end{equation}
   We argue by recurrence on $k-j$ and notice that the case $k=j$ is trivial since $\widetilde R_M^{(j)}(z_j)=\widetilde f^{(j)}_j$ by definition. In case $k>j$ we may write
   \begin{eqnarray*} &&
       \widetilde f^{(j)}_k - \widetilde R_M^{(j)}(z_k) =
       \widetilde f^{(j)}_k -\widetilde f^{(j)}_j - (\widetilde R_M^{(j)}(z_k) -\widetilde f^{(j)}_j)
       \\&&  =
       \frac{z_k-z_j}{\widetilde f^{(j+1)}_k} (1+\epsilon_{j+1,k})
        - (\widetilde R_M^{(j)}(z_k) -\widetilde f^{(j)}_j)
        =
       \Bigl( \frac{1+\epsilon_{j+1,k}}{1+\delta_{j+1,k}} - 1 \Bigr)
        (\widetilde R_M^{(j)}(z_k) -\widetilde f^{(j)}_j).
   \end{eqnarray*}
   Our claim \eqref{T_claim} then follows by observing\footnote{Without this positivity assumption, \cite[Theorem 4.1]{GM80} observed with \eqref{T_pivot} that, up to $\mathcal O(\varepsilon)$, we have that $|\widetilde R_M^{(j)}(z_k) -\widetilde f^{(j)}_j|\approx |\widetilde f^{(j)}_k -\widetilde f^{(j)}_j| \leq |\widetilde f^{(j)}_k| + |\widetilde f^{(j)}_j|
   \leq 2 |\widetilde f^{(j)}_k|\approx 2 \,|\widetilde R_M^{(j)}(z_k)|$, leading to some exponentially increasing growth factor.} that $|\widetilde R_M^{(j)}(z_k) -\widetilde f^{(j)}_j|=\widetilde R_M^{(j)}(z_k) -\widetilde R_M^{(j)}(z_j)\leq \widetilde R_M^{(j)}(z_k)$ by assumption $\widetilde f^{(j)}_j>0$ for $j=1,...,M$, and by the inequality
   $$
        | \frac{1+\epsilon_{j+1,k}}{1+\delta_{j+1,k}} - 1 | \leq
        \frac{|\epsilon_{j+1,k}| + \gamma_{k-j-1}}{1-\gamma_{k-j-1}} \leq \gamma_{k-j}.
   $$
   In a similar manner, we deduce the assertion of the Theorem from \eqref{T_claim} for $j=1$.
\end{proof}

\begin{figure}[!h]
	\centering
%	\begin{minipage}[c]{1\textwidth}
		\includegraphics[width=1\linewidth]{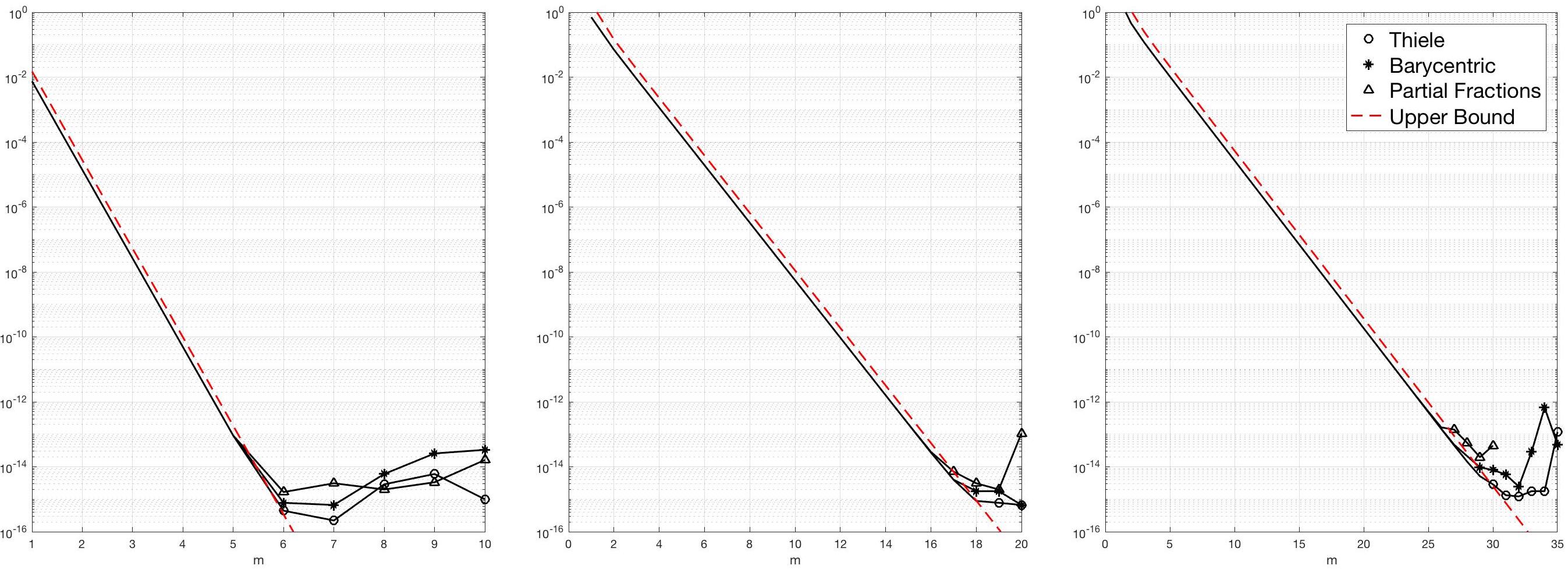}
		\caption{Relative $L^\infty$ error on the interval $[c,d]$ of rational interpolants of type $[m-1|m]$ of the Markov function $f(z)=1/\sqrt{z}$, with $\alpha=-\infty,\beta=0$, $d=1$ and $c\in \{ 1/2,10^{-3},10^{-6}\}$ (from the left to the right). For each $c$ and $m$, we take the quasi-optimal interpolation points of \eqref{optimal_nodes} (depending on $m$ and $\alpha,\beta,c,d$), and show the relative error of the same rational interpolant (black solid line), computed with three different methods: the partial fraction decomposition of \S\ref{sec3.1} (triangle markers), the barycentric representation of \S\ref{sec3.2} (star markers) and finally the Thiele interpolating continued fraction of \S\ref{sec3.3} (circle markers). The fourth graph (red dashed) gives the a priori upper bound \eqref{optimal_nodes2} of Corollary~\ref{cor_a_priori}. }\label{fig_representation1}
%	\end{minipage}
\end{figure}

\begin{Rem}\label{Rem_finite_precision}
   Extensive numerical experiments showed us that the parameters $\widetilde f^{(j+1)}_k$ of Theorem~\ref{Thm_finite_precision} only fail to be positive if the error $R_j^{(1)}(z)-f^{(1)}(z)$ for $z\in [c,d]$ is already close to machine precision. To prove such a statement, one requires a (rough) forward stability result on $\widetilde f_k^{(j)}-f_k^{(j)}$ which seems to be possible but quite involved, we omit details.
%   In our implementation of the Thiele interpolating continued fraction with pivoting \eqref{T_pivot} we stop the computations and return $\widetilde R_j^{(1)}$ if one of the parameters $\widetilde f^{(j+1)}_k$ becomes negative or if a permutation occurs due to pivoting.
   %\footnote{Jene suis pas sur de cette condition d'arret. Peut-etre il faudra essayer l'alternative que l'on applique le pivotage \eqref{T_pivot}, et on affiche si le pivotage a ete vraiment necessaire, car dans ce cas on passe ne notre backward stability a celle de Graves-Morris, avec un facteur de croissance exponentiel.}
\end{Rem}

\begin{Ex}
   In Figure~\ref{fig_representation1} we represent the relative $L^\infty([c,d])$ error of the same interpolants $r_m$ for the same Markov function $f(z)=1/\sqrt{z}$ and interpolation points \eqref{optimal_nodes} depending on $m$, computed with the three different methods discussed so far. Here we have discretized $[c,d]$ by $500$ cosine points, the entries of some diagonal matrix $A$. Recall that, in exact arithmetic, all curves should have identical behavior, and stay below the a priori upper bound \eqref{optimal_nodes2}. However, in finite precision arithmetic we observe that, once the method and the value of $c$ is fixed, the corresponding error polygon crosses the upper bound once and, afterwards, does hardly decrease, and sometimes even increases. We use markers on the error curves for indices $m$ which have been rejected by our stopping criterion of Remark~\ref{rem_stopping}. If we denote by $m'$ the index such that $m'+1$ is the first rejected index, the error of $m'$ is below the a priori bound, and the crossing happens between the indices $m'$ and $m'+1$. Also, we observe without theoretical evidence that the error for any $m>m'$ is never smaller than $1/10$ times the error for $m'$. This confirms that our stopping criterion works well in practice. Notice that, for any of the three methods, the final relative error is about the same size (not far from machine precision), and increases only modestly with $d/c$. In Section~\ref{sec_experiences} we will see that such a behavior is no longer true if we evaluate our interpolants at general matrix arguments instead of scalar arguments.
\end{Ex}

%\textcolor{red}{From Figure~\ref{fig_representation1} we see that we should not use the partial fraction approach of \S\ref{sec3.1}, and that the Thiele approach of \S\ref{sec3.3} seems to be the most reliable. In addition, our stopping criterion strategy of Remark~\ref{rem_stopping} finds automatically the index $m$ leading to smallest relative error.}

%\begin{Rem}\label{stopping} \ref{cor_posteriori}
%   Also, a stopping strategy is necessary for our barycentric approach of \S\ref{sec3.2} and the Thiele continued fraction of \S\ref{sec3.3}. Such a stopping criterion can be deduced from the first part of Corollary~\ref{cor_posteriori} together with \\eqref{optimal_nodes2}, we omit details here, since details are given later in for the particular that residual bound of Corollary Such a stopping criterion cannot be only based in our a priori upper bound, unless we are aiming for low precision. We do not have a satisfying answer for the barycentric case. However, the Thiele polygons cross the a priori bound for $m\geq 7,17,29$ (from the left to the right), which exactly is the set of indices where the numerical parameters are no longer $>0$, or a pivoting was necessary. Thus, instead of returning a lower degree rational function as suggested in Remark~\ref{Rem_finite_precision}, we should stop to compute $r_m$ for $m\geq M$ once a wrong sign or a permutation is detected for index $m=M$. Though we have no theoretical evidence, such a stopping criterion seems to work well in practice, also for other examples.
%\end{Rem}}

It still remains to present a proof of Theorem~\ref{Thm_Markov} which will be based on the following Lemma which is partly known from the classical Stieltjes moment problem up to a change of variables, see for instance \cite[Sections 5.2 et 5.3]{BGM96} or \cite[Thm V.4.4]{bra86}. In the remainder of this section we suppose that $\alpha<\beta <z_0$. For a function $g$ analytic in some neighborhood of $z_0$ the Hankel matrices are defined with help of the Taylor coefficients of $g$ at $z_0$
\begin{equation} \label{M_Hankel}
    \mathcal H_n^{(\ell)}(g) =
    \left[\begin{array}{cccc} g_{\ell} & g_{\ell+1} & \cdots & g_{n+\ell}
    \\ g_{\ell+1} & g_{\ell+2} & \cdots & g_{n+\ell+1}
    \\ \vdots & \vdots & & \vdots \\
    g_{n+\ell} & g_{n+\ell+1} & \cdots & g_{2n+\ell}
\end{array}\right] , \quad g(z) = \sum_{j=0}^\infty g_j (z-z_0)^j .
\end{equation}
The following lemma will be applied for $z_0\in \{ z_1,z_2,...\}$ with the $z_1,z_2,..$ as in \eqref{ordering}.

\begin{Lemme}\label{Lem_moments}
    If $f$ is a Markov function with measure $\mu$ having an infinite support included in $[\alpha,\beta]$ then, for all $n\geq 0$,  the Hankel matrices $\mathcal H_n^{(0)}(f)$ are positive definite, and the Hankel matrices $\mathcal H_n^{(1)}(f)$ are negative definite. \\
    Conversely, if $f$ is analytic in $\mathbb C \setminus [\alpha,\beta]$, with Hankel matrices $\mathcal H_n^{(0)}(f)$ positive definite and $\mathcal H_n^{(1)}(f)$ negative definite for all $n\geq 0$, then $f$ is a Markov function with measure $\mu$ having an infinite support included in $[\alpha,\beta]$.
\end{Lemme}
\begin{proof}
    A proof of the first part is elementary, noticing that the Taylor coefficients are moments of $\mu$ given by $f_j = \int \frac{d\mu(x)}{(z_0-x)(x-z_0)^j}$, leading to an integral expression of $y^T \mathcal H_n^{(\ell)}(f)y$ for any $y\in \mathbb R^{n+1}$ with a unique sign depending only on the parity of $\ell$ provided that $y\neq 0$, for details see \cite{these_Bisch}.\\
    To show the converse implication, one considers $r_m=p_m/q_m$ being the Padé approximant of type $[m-1|m]$ of $f$ at $z_0$. The sign assumption on the Hankel determinants allows to conclude that $q_m$ has a determinant representation $q_m(z)=\det\Big( \mathcal{H}^{(0)}_{m-1}(f)-(z-z_{0})\mathcal{H}^{(1)}_{m-1}(f) \Big)$. Moreover, there is a three term recurrence between three consecutive denominators with the sign of the coefficients being known. An additional Sturm sequence argument allows us to conclude that $r_m=p_m/q_m$ has $m$ distinct poles $x_{1,m},...x_{m,m}\in (-\infty,z_0)$ and positive residuals $a_{j,m}$, that is
    \begin{equation}\label{finite_measure}
         r_m(z) = \frac{p_m(z)}{q_m(z)} = \int \frac{d\mu_m(x)}{z-x} , \quad \mu_m = \sum_{j=1}^m a_{j,m} \delta_{x_{j,m}} .
    \end{equation}
    As in \cite[Proof of Thm. V.4.4]{bra86}, there exists a subsequence $(\mu_{m_\ell})_\ell$ of $(\mu_m)_m$ having the weak-star limit $\widetilde \mu$, $\supp(\widetilde \mu)\subset (-\infty,z_0]$, and for all $k \geq 0$
    $$
           \lim_{\ell \to \infty} \int \frac{d\mu_{m_\ell}(x)}{(z_0-x)^{k+1}}
           = \int \frac{d\widetilde \mu(x)}{(z_0-x)^{k+1}} = (-1)^k \frac{g^{(k)}(z_0)}{k!},
    $$
    with the Markov function $g(z)=\int\frac{d\widetilde \mu(x)}{z-x}$.
    From the interpolation conditions of a Padé approximant of $f$ at $z_0$ we also know that, for $k \leq 2m$, 
    $$
            \int\frac{d\mu_m(x)}{(z_0-x)^{k+1}} =
              (-1)^k \frac{r_m^{(k)}(z_0)}{k!} = (-1)^k \frac{f^{(k)}(z_0)}{k!}.          
    $$
    Combining these two relations we find that $g^{(k)}(z_0)$ is finite, and 
    $g^{(k)}(z_0)=f^{(k)}(z_0)$ for all $k\geq 0$. In particular, with $f$ also $g$ is analytic in a neighborhood $U$ of $z_0$, and $f(z)=g(z)$ for all $z\in U$. 
    Recalling that by assumption $f$ is analytic in $\mathbb C \setminus [\alpha,\beta]$, we see that our Markov function $g$ for the measure $\widetilde \mu$ has an analytic continuation $f$ in $\mathbb C \setminus [\alpha,\beta]$, and hence $\supp(\widetilde \mu)\subset [\alpha,\beta]$. Thus also the converse implication is true. 
%%    . We may conclude that $p_m/q_m$ tends to the Markov function $g(z)=\int\frac{d\widetilde \mu(x)}{z-x}$ on any compact subset of $\mathbb C \setminus (-\infty,z_0]$.
    %\footnote{Ici j'ai abrégé pas mal, le raisonnement me semble etre assez proche du raisonnement connu de Stieltjes. Mais il me semble que l'on n'avait pas demontre que $f=g$, ce qui est le but des considerations suivantes...}
%%    We now observe that, for any $\gamma>z_0$ and $k\geq 0$,
 %%   \begin{eqnarray*}
 %%             0 &<& (-1)^k \frac{g^{(k)}(\gamma)}{k!} 
%%              = \int\frac{d\widetilde \mu(x)}{(\gamma-x)^{k+1}} 
%%              = \lim_{m\to \infty} \int\frac{d\mu_m(x)}{(\gamma-x)^{k+1}} 
%%              \\          &\leq& \lim_{m\to \infty} \int\frac{d\mu_m(x)}{(z_0-x)^{k+1}} =
%%              \lim_{m\to \infty}  (-1)^k \frac{r_m^{(k)}(z_0)}{k!} = (-1)^k \frac{f^{(k)}(z_0)}{k!},
%%    \end{eqnarray*} 
 %%   the last equality following from the interpolation condition of a Padé approximant of $f$ at $z_0$. By assumption on $f$, the $k$th root of the expression on the right has a $\lim\sup>0$ not depending on the choice of $\gamma$. By choosing $\gamma$ sufficiently close to $z_0$ we conclude that $g$ is analytic in a neighborhood of $z_0$, and hence $z_0\not\in \supp(\widetilde \mu)$. 
%%    Finally, the more precise inclusion $\supp(\mu)\subset [\alpha,\beta]$ is a consequence of the fact the Markov function $g$ cannot have an analytic continuation in a neighborhood of a point in $\supp(\mu)$. We again refer the reader to \cite{these_Bisch} for further details.
\end{proof}

\begin{proof}[Proof of Theorem~\ref{Thm_Markov}.]
   We only need to show this statement for $j=1$. Let $f^{(1)}=1/f$, with $f$ a Markov function with measure $\mu$ having an infinite support included in $[\alpha,\beta]$.
   Since $f(z)\neq 0$ for $z\not\in [\alpha,\beta]$, we conclude that $f^{(1)}$ is analytic in $\mathbb C \setminus [\alpha,\beta]$, and that the same is true for
   $$
        g(z)=\frac{f^{(1)}(z)-f^{(1)}(z_1)}{z-z_1} .
   $$
   Moreover, since $f$ is non-real in $\mathbb C \setminus \mathbb R$ and strictly decreasing in $\mathbb R \setminus [\alpha,\beta]$, we also observe using \eqref{T_functions} that $f^{(2)}=1/g$
   is analytic in $\mathbb C\setminus [\alpha,\beta]$. As a consequence of the first part of Lemma~\ref{Lem_moments} applied to $f$, the Hankel matrices
   $$
          \mathcal H_n^{(0)}(\frac{1}{f^{(1)}}) = \mathcal H_n^{(0)}(f) , \quad \mbox{and} \quad
          - \mathcal H_n^{(1)}(\frac{1}{f^{(1)}}) = - \mathcal H_n^{(1)}(f)
   $$
   are positive definite for all $n\geq 0$. According to  the second part of Lemma~\ref{Lem_moments} applied to $g=1/f^{(2)}$, it only remains to show that the Hankel matrices
   $$
          \mathcal H_n^{(0)}(\frac{1}{f^{(2)}}) = \mathcal H_n^{(1)}(f^{(1)}) , \quad \mbox{and} \quad
          - \mathcal H_n^{(1)}(\frac{1}{f^{(2)}}) = - \mathcal H_n^{(2)}(f^{(1)})
   $$
   are positive definite for all $n\geq 0$. The latter is a consequence of the Hadamard bigradient identity \cite[Thm 2.4.1]{BGM96}:
   $$
            \det H_n^{(m)}(f^{(1)}) = (-1)^{(n+1)+(m-1)(m-2)/2} f^{(1)}(z_1)^{m+2n+1} \det H_{m+n-1}^{(2-m)}(\frac{1}{f^{(1)}}),
   $$
   since then
   \begin{eqnarray*} &&
       \det \mathcal H_n^{(0)}(\frac{1}{f^{(2)}}) = \det \mathcal H_n^{(1)}(f^{(1)})
       = (-1)^{n+1} f^{(1)}(z_1)^{2n+2} \det H_{n}^{(1)}(\frac{1}{f^{(1)}}) > 0 ,
       \\&&
       (-1)^{n+1} \det \mathcal H_n^{(1)}(\frac{1}{f^{(2)}}) = (-1)^{n+1} \det \mathcal H_n^{(2)}(f^{(1)})
       = f^{(1)}(z_1)^{2n+3} \det H_{n+1}^{(0)}(\frac{1}{f^{(1)}}) > 0 ,
   \end{eqnarray*}
   as required to conclude.
\end{proof}

    To summarize, in Section~\ref{sec3} we addressed the question how to represent and compute the rational interpolant $r_m$ for given real interpolation nodes in finite precision arithmetic (since, as reported in \cite[Sections~2.1]{BGM96}, a naive implementation might lead to a loss of at least $m$ decimal digits of precision). Here we compare three approaches: firstly the partial fraction decomposition of \S\ref{sec3.1} promoted by Mayo and Antoulas \cite{MAYO07} where the poles are the eigenvalues of a Loewner matrix pencil and the residuals are found through a least square problem. Secondly we analyze the barycentric representation of \S\ref{sec3.2} which recently \cite{BB18} has been used quite successfully for stabilizing the rational Remez algorithm, and for which backward and forward results are known for evaluating such rational functions in finite precision arithmetic. Finally we consider the Thiele interpolating continued fraction in \S\ref{sec3.3} which generalizes the concept of Stieljes continued fraction representation of Padé approximants of Markov functions.  Our main original contributions in this section are Theorem~\ref{Thm_Markov} showing that parameters of the Thiele interpolating continued fraction of a Markov function are positive, and Theorem~\ref{Thm_finite_precision} where we provide a proof of backward stability of Thiele interpolating continued fractions improving a result of Graves-Morris \cite[Theorem 4.1]{GM80}.
    Numerical experiences presented in Figure~\ref{fig_representation1} show that any of these three methods combined with our stopping criterion of Remark~\ref{rem_stopping} allows to attain nearly machine precision for scalar arguments, but this will be no longer true for matrix arguments.

\section{Functions of Toeplitz-like matrices}\label{sec4}

In the last years, several authors tried to take advantage of structure in a square matrix $A$ in order to speed up the approximate computation of matrix functions $f(A)$. One possible approach is to consider algebras of structured matrices as for instance hierarchical matrices in HODLR or HSS format which are closed under addition, multiplication with a scalar and inversion, and contain the identity, see for instance the recent paper \cite{MRK20}
%{\footnote{S. Massei, Leonardo  Robol, D. Kressner, hm-toolbox:  Matlab software for HODLR and HSS matrices, arXiv:1909.07909v3 (2020).}
and the references therein. The hierarchical rank $k$ of $A\in \mathbb R^{n\times n}$ gives a complexity parameter and, roughly speaking, the above matrix operations can be carried out in $\mathcal O(k^2n)$ or $\mathcal O (k^2n \log(n))$ operations, see \cite[Table in \S 4.3]{MRK20}.
Replacing $f$ by a rational function $r$ of type $[m-1|m]$ and evaluating $r(A)$ within the hierarchical algebra following the operations described in \S \ref{sec3.1}, \S\ref{sec3.2}, or \S\ref{sec3.3}, requires to compute about $2m$ shifted inverses, and in the worst case might increase the hierarchical rank from $k$ (for $A$) to $2mk$ (for $r(A)$). It is therefore important to know that the above operations are combined with a compression procedure of the same complexity, in order to keep the hierarchical rank as small as possible.

Another structural property allowing for the approximate computation of $f(A)$ in low complexity %that can assert a low-rank $f(A)$
is {\itshape displacement structure}, which was discovered independently by Heinig and Rost \cite{HR84,Heinig1989}, and in a series of works by Kailath and others, e.g.,~\cite{Kail79,Kail91,Kail95,Pan93}.  Our inspiration to use this kind of structure was a work of Kressner and Luce~\cite{Kress18}, who used displacement structure for the fast computation of the matrix exponential for a Toeplitz matrix.

The displacement operator that we will use here is the so-called
{\itshape Sylvester} displacement operator $S: \mathbb C^{n \times n}
\rightarrow \mathbb C^{n \times n}$, defined by
\begin{equation} \label{Displacement_def}
    S(A) = Z_{1} A - A Z_{-1} ,\quad \text{where} \;
Z_{\theta}=\begin{pmatrix}
0 & \ldots & \ldots & 0 & \theta \\
1 & \ddots & & & 0 \\
0 & \ddots & \ddots &  & \vdots \\
\vdots & \ddots & \ddots & \ddots & \vdots \\
0 & \ldots & 0 & 1 & 0
\end{pmatrix}.
\end{equation}
The number $\tau=\tau(A):=\rank(S(A))$ is called the {\itshape displacement rank of $A$}, and any pair of matrices $G, B \in \mathbb C^{n\times \tau}$ such that $S(A) = GB^*$ is called {\itshape a generator} for $S(A)$. It is readily verified that the displacement rank of a Toeplitz matrix \eqref{toeplitz} is at most two, and the same is true for the shifted matrix $A - zI$ (being a Toeplitz matrix itself), and also for the resolvent $(zI - A)^{-1}$.  It is customary to say that a matrix $A \in \mathbb C^{n \times n}$ is {\itshape Toeplitz-like}, if its displacement rank is ``small'', i.e., if $\tau = \rank(S(A)) \ll n$.

We will now briefly review how the displacement rank behaves under elementary operations.  Consider two (Toeplitz-like) matrices $A_1, A_2$ having displacement ranks $\tau_1$ and $\tau_2$, respectively.
\begin{itemize}
    \item The identity matrix has displacement rank one.
    \item For a scalar $0 \neq s\in \mathbb C$ we have $\tau(s A_1) = \tau_1$.
    \item For sums of Toeplitz-like matrices we have $\tau(A_1 + A_2) \le \tau_1 + \tau_2$.
    \item For products of Toeplitz-like matrices we have $\tau(A_1 A_2) \le \tau_1 + \tau_2 + 1$.
\end{itemize}
It is interesting to note that for many operations it is possible to efficiently compute a generator of the result directly from the generators of the operands.  For example, for a Toeplitz-like matrix $A$ with generator $(G, B)$ one finds that $S(A^{-*}) = -(Z_1 A^{-*} B)(Z^*_{-1} A^{-1} G)^*$, implying that a generator for $A^{-*}$ can be obtained by solving two linear system with $A^*$ and $A$, respectively (and that the displacement rank does not increase here either).  An exhaustive description and discussion of operations among Toeplitz-like matrices, and the effects on the displacement rank is given in~\cite{these_Bisch}.

In order to understand the arithmetic complexity of evaluating a rational function $r(A)$ of a Toeplitz-like matrix $A$, we will need to evaluate matrix-vector products with $A$, and solve linear systems of equations with $A$.  A matrix-vector product with a $A$ (or $A^*$) can be computed in $\mathcal O(\tau n \log(n)$), using the FFT.  The currently best asymptotic complexity for solving linear systems of equations with a Toeplitz-like matrix is in $\mathcal O(\tau^2 n \log^2(n))$\footnote{This algorithm is based on transforming via FFT a Toeplitz-like  matrix into a Cauchy-like matrix being a hierarchical matrix of hierarchical rank $k=\mathcal O(\tau \log(n))$, and then to use the hierarchical solver of {\tt hm-toolbox}, see \cite[\S 5.1 and Table in \S 3.4]{MRK20} and \cite{Xia12}.
%(Ajouter reference originale J. Xia, Y. Xi, and M. Gu.  A superfast structured solver for Toeplitz linear systems via randomized sampling. SIAM J. Matrix Anal. Appl., 33(3):837–858, 2012.).
}, plus a compression procedure based on a singular value decomposition of $S(A)$ in complexity $\mathcal O (\tau^2 n)$ where we drop contributions from singular values below the machine precision.  Again, we refer to the thesis~\cite{these_Bisch}, where the details and further references are spelled out.

In our numerical experiments reported in \S\ref{sec_experiences}, we use the ``TLComp'' Matlab toolbox\footnote{\url{https://github.com/rluce/tlcomp}}, which offers an automatic dispatch of operations with Toeplitz-like matrices to implementations that work directly with generators (as opposed to the full, unstructured matrix).  Note that in contrast to the best possible asymptotic complexity pointed out above, this toolbox solves linear systems of equations using the GKO algorithm~\cite{GKO95}, which (only) has a complexity in $\mathcal O(\tau n^2)$.  Extensive numerical experiments have shown, however, that for the practical range of dimension $n \in [10^3,10^5]$ under consideration here, the classical GKO algorithm turns out to be much faster, which is why we stick to it in our numerical experiments.  Of course, we still phrase complexity results in Theorem~\ref{Thm_computation} below with respect to the better asymptotic bound.  This toolbox also allows the fast computation (or approximation) of various norms of Toeplitz-like matrices, and the reconstruction of the full matrix $A$ from its generators in (optimal) complexity $\mathcal O(\tau n^2)$.  A complete description of ``TLComp'' and its functionality will be subject of a future publication.

In all our experiments we considered only real symmetric Toeplitz and Toeplitz-like matrices with spectrum in a given interval $[c,d]$, where the different operations even simplify, and we have the error estimates
$$
       \| f(A) - r_m(A) \| \leq \| f - r_m \|_{L^\infty([c,d])} , \quad
       \| I - r_m(A)f(A)^{-1} \| \leq \| \frac{f - r_m}{f} \|_{L^\infty([c,d])} .
$$
Notice that a priori there is no reason to expect that $f(A)$ has a small displacement rank, even for our special case of a Markov function $f=f^{[\mu]}$. However, for a rational function $r_m^{[\mu]}$ of type $[m-1|m]$, the displacement rank of $r^{[\mu]}_m(A)$ is at most $\mathcal O(m (\tau+1))$. Also, we will always choose rational interpolants $r^{[\mu]}_m$ with interpolation nodes given by \eqref{optimal_nodes}, which according to \eqref{optimal_nodes2} allows us to achieve precision $\delta>0$ for $m=\mathcal O(\log(1/\delta))$, with the hidden constant depending only on the cross ratio of $\alpha,\beta,c,d$. Indeed, in all experiments reported in \S\ref{sec_experiences}, the displacement rank of $r_m^{[\mu]}(A)$ (after compression) grows at most linearly with $m$, and sometimes even less if the precision increases.

We summarize our findings in the following theorem, in its proof we also discuss the different implementations of our three approaches of \S 3 in the algebra of Toeplitz-like matrices.

\begin{Thm}\label{Thm_computation}
    Let $ f^{[\mu]}:z\mapsto \inalbe\frac{d\mu(x)}{z-x} $ be a Markov function with $\supp(\mu)\subset [\alpha,\beta]$, $\delta>0$, $m \geq 1$, $A\in \mathbb R^{n\times n}$ a symmetric Toeplitz-like matrix with displacement rank $\tau$, and spectrum included in the real interval $[c,d]$, with $c>\beta$. Furthermore, denote by $r^{[\mu]}_m$ the rational interpolant of $f^{[\mu]}$ of type $[m-1|m]$  (in exact arithmetic) at the interpolation nodes \eqref{optimal_nodes} (depending only on $m$ and $\alpha,\beta,c,d$). Then for $m=\mathcal O(\log(1/\delta))$, $r^{[\mu]}_m(A)$ of displacement rank $\mathcal O(m \tau)$ is an approximation of $f^{[\mu]}(A)$ of (relative) precision $\mathcal O(\delta)$.
    \\
    Furthermore, computing the generators of $r_m^{[\mu]}(A)$ through the techniques of \S\ref{sec3.1}, \S\ref{sec3.2}, and \S\ref{sec3.3} within the algebra of Toeplitz-like matrices has complexity
    $\mathcal O (m \tau^3 n \log^2(n))$ for the first two approaches, and
    $\mathcal O (m^2 \tau^3 n \log^2(n))$  for the Thiele continued fraction.
\end{Thm}
\begin{proof}
    It only remains to show the last part, where we ignore the cost of computing poles/residuals or other parameters which is of complexity $O(m^3)$. The partial fraction decomposition \eqref{pfd} in \S\ref{sec3.1} seems to be the easiest approach to compute the generators of $r^{[\mu]}_m(A)$: we just have to compute the generators of each resolvent $(A- x_j I)^{-1}$ (with displacement rank bounded by $\tau+1$), combine and compress. Here the essential work is to compute $m$ times the generator of a resolvent, by solving at most $2m(\tau+1)$ systems of Toeplitz-like matrices, leading to the claimed complexity.

    The barycentric representation \S\ref{sec3.2}  requires to compute separately the generators of
    $$
         P(A) = \sum_{j=0}^{m} f(t_j) \beta_{j} (A-t_jI)^{-1}, \, \,
         Q(A) = \sum_{j=0}^{m} \beta_{j} (A-t_jI)^{-1}
    $$
    of displacement rank at most $(m+1)(\tau+1)$, then those of $Q(A)^{-1}$ and finally those of $P(A)Q(A)^{-1}$ with a cost being about 4 times the one discussed before.

    Finally, for insuring stability in \S\ref{sec3.3}, we use the backward evaluation of $R_{2m}^{(1)}(A)$ via \eqref{T_evaluation}, leading to $R_{2m}^{(2m)}(A)=f_{2m}^{(2m)} I $, and $R_{2m}^{(j)}(A) = f_{j}^{(j)} I +
    (A-z_j I) R^{(j+1)}_{2m}(A)^{-1}$ for $j=2m-1,2m-2,..., 1$. Here the cost is dominated by finding the generators of the inverse of $R^{(j+1)}_{2m}(A)$, of displacement rank at most $(\tau+1)(2m+2-j)/2$.
\end{proof}

    To summarize, in \S\ref{sec4} we have reported about how to efficiently evaluate $r^{[\mu]}_m(A)$ for a Toeplitz matrix $A\in \mathbb R^{n\times n}$ in complexity $\mathcal O(n \log^2(n))$ (with the hidden constant depending on $m$ and the desired precision), by exploiting the additional Toeplitz structure. This part has been strongly inspired by previous work of Kressner and Luce \cite{Kress18}, see also \cite{MRK20}, who exploited the theory of small displacement rank and Toeplitz-like structured matrices. Let us illustrate these findings by some numerical experiments.
    %, see for instance \cite{Heinig84,Kail79,Kail91}.

\section{Numerical experiments}\label{sec_experiences}

In this section, we illustrate our findings for Toeplitz matrices by reporting several numerical experiments. In all figures to follow, for a fixed symmetric positive definite Toeplitz matrix $A$ of order $500$ with extremal eigenvalues $\lambda_{\min},\lambda_{\max}$ and a fixed Markov function $f^{[\mu]}$ recalled in the caption, we present four different cases (displayed from the left to the right)
\begin{equation} \label{4experiences}
   \begin{array}{ll}
    (i) & \mbox{Toeplitz-like arithmetic, $[c,d]=[\lambda_{\min},\lambda_{\max}]$,}
\\
    (ii) & \mbox{Toeplitz-like arithmetic, $[c,d]=[\frac{1}{2}\lambda_{\min},2\lambda_{\max}]$,}
\\
    (iii) & \mbox{without Toeplitz-like arithmetic, $[c,d]$ as in $(i)$,}
\\
    (iv) & \mbox{$[c,d]$ as in $(i)$, diagonal matrix of order $500$ with entries being cosine points in $[c,d]$.}
\end{array}
\end{equation}
Hence, the impact of an enlarged spectral interval can be measured in comparing cases $(i)$ and $(ii)$, the impact of our particular Toeplitz-like arithmetic by comparing cases $(i)$ and $(iii)$, and finally the impact of a matrix-valued argument instead of a scalar argument (again due to finite precision arithmetic) by comparing the cases $(i)--(iii)$ with $(iv)$.  For each of the four cases, we show the relative error of the same rational interpolant of type $[m-1|m]$ with the quasi-optimal interpolation points of \eqref{optimal_nodes} depending on $m$ (black solid line). Due to finite precision arithmetic, we obtain three error curves: the partial fraction decomposition of \S\ref{sec3.1} (triangle markers), the barycentric representation of \S\ref{sec3.2} (star markers) and finally the Thiele interpolating continued fraction of \S\ref{sec3.3} (circle markers). As in Figure~\ref{fig_representation1}, we use a marker for index $m$ if \eqref{eq_stopping} holds, in other words, this index is rejected by our stopping criterion of Remark~\ref{rem_stopping}.
The forth graph (red dashed) gives the a priori upper bound \eqref{optimal_nodes2} of Corollary~\ref{cor_a_priori}. Notice that computing the relative error requires to evaluate $f(A)$ via built-in matrix functions of Matlab (such as {\tt logm()}), which explains why we consider only matrices of moderate size.

\begin{Ex}\label{Ex5.1}
   We present a first numerical example where we approach $\log(A)$, with $A$ being a symmetric positive definite Toeplitz matrix of order $500$, with extremal eigenvalues $\lambda_{\min}=25,\lambda_{\max}=139.2$ and a condition number $5.568$, created by using random generators and shifting the spectrum.  Notice that $f(z)=\log(z)/(z-1)$ is a Markov function with $\alpha=-\infty$ and $\beta=0$. Denoting by $r_m$ a rational interpolant of type $[m-1|m]$ of $f$, we will thus approach $\log(A)$ by $(A-I)r_m(A)$, leading to the same relative error as approaching $f(A)$ by $r_m(A)$. %In order to display the relative errors in each of the four cases of Figure~\ref{fig6}, we have computed explicitly $\log(A)$ via the Matlab function {\tt logm} (and thus we deal with matrices of moderate size).
   The four different cases (from left to the right) are those explained above in \eqref{4experiences}.
    Some observations are in place:
    \begin{enumerate}
    \item[(a)] In any of the four cases and 3 methods, the stopping criterion of Remark~\ref{rem_stopping} works perfectly well: all accepted indices give errors below our a priori bound being nicely decreasing for increasing $m$. Also, all rejected indices correspond to errors above our a priori bound, and these errors are never much smaller than the error at the last accepted index.
    \item[(b)] In the case $(iv)$ of diagonal matrices $A$ (or, equivalently, for scalar arguments), all three methods for evaluating $r_m(A)$ are equivalent, this confirms similar observations in Figure~\ref{fig_representation1}.
    \item[(c)] In any of the cases $(i)-(iii)$, that is, for full matrices $A$, the barycentric representation of $r_m(A)$ leads to much larger errors, in particular if one uses Toeplitz-like arithmetic.
    \item[(d)] The partial fraction decomposition and the Thiele continued fraction approaches have a similar behavior, and lead to a small error of order $10^{-12}$.
    \item[(e)] Enlarging the spectral interval as in case $(ii)$ does not lead to a smaller error, but might require to compute interpolants of higher degree for achieving the same error.
   \item[(f)] For measuring the complexity, it is also interesting to observe that the displacement rank of $r_m(A)$ in the cases $(i)$ and $(ii)$ is increasing in $m$, it first increases linearly, and then stabilizes around $22$ (for Thiele and partial fractions, the double for barycentric), once a good precision is reached.
    \end{enumerate}
%    notice that all error curves cross the a priori upper bound due to finite precision. Comparing the three approaches of \S 3, it seems that the partial fraction approach always fails first, and that the Thiele continued fraction approach nearly always gives the best results, here a relative precision between $10^{-13}$ and $10^{-15}$, depending on the case. We notice that the barycentric approach which in \S 3 gave satisfactory results in the case of scalar arguments (or, what amounts the same, for diagonal matrices as displaced in case $(iv)$ on the right), might also fail for Toeplitz matrix arguments.
%   We also observe that the error curves get more erratic for Toeplitz-like arithmetic compared to the case $(iii)$ where we disregard the structure, once they cross the a priori upper bound. A phenomenon we observed for other matrices is that enlarging the spectral interval (case (ii)) sometimes smoothes this erratic behavior, but does non always lead to a smaller final error. Notice also that our residual stopping criteria of Remark~\ref{rem_stopping} works surprisingly well, no matter how we computed $r_m(A)$: on each error curve, the error is displayed by a marker if the condition \eqref{eq_stopping} is satisfied. Thus, following Remark~\ref{rem_stopping}, for example for the case $(i)$, we should stop at $m=8$ for Thiele, and $m=6$ for the other two methods, each time the index where the smallest error is reached.
%
\end{Ex}

\begin{figure}[!h]
	\begin{center}
		\includegraphics[width=0.95\linewidth]{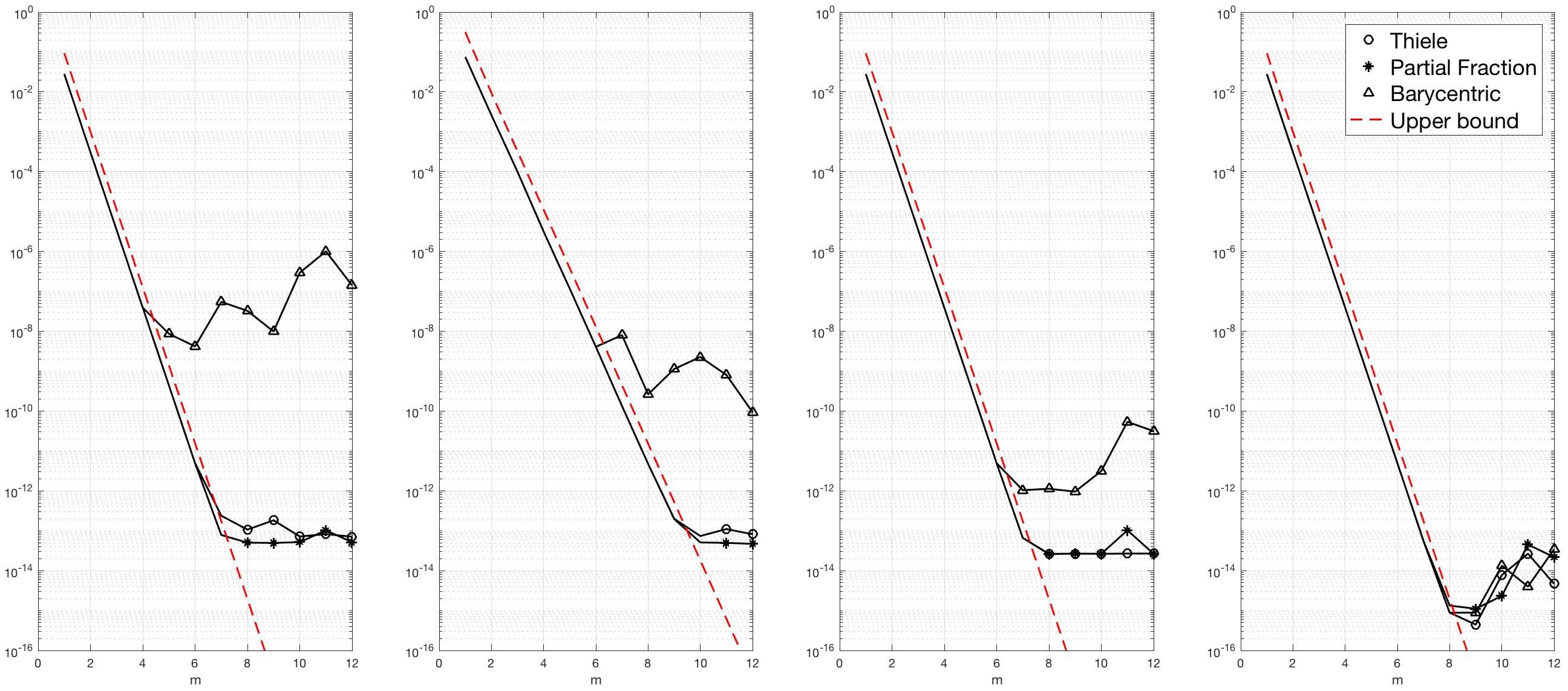}
    \end{center}
		\caption{Relative errors for approaching $f(A)$ for the Markov function $f(z)=\log(z)/(z-1)$, and a symmetric positive definite Toeplitz matrix of order $500$, with extremal eigenvalues $\lambda_{\min}=25,\lambda_{\max}=139.2$ and a condition number $5.568$. Further explanations on the legend and the four different cases $(i)-(iv)$ (displayed from the left to the right) are given in the first paragraph of \S\ref{sec_experiences}.
%We present four different experiments (displayed from the left to the right) denoted by $(i)-(iv)$, and  explained in \eqref{4experiences}. For each experiment, we display (black solid lines) the relative error at our matrix argument of the rational interpolants $r_m$ as a function of the denominator degree $m$, and interpolation points \eqref{optimal_nodes}. Here we show the results of the three approaches of \S 3, namely partial fraction decomposition (marked by stars), the barycentric representation (marked by triangles), and the Thiele continued fraction (marked by circles). We also displayed (red dashed line) the a priori upper bound \eqref{optimal_nodes2}. Notice that the error curves are nearly identical as long as they are below the a priori bound. On each error curve, the error is displayed by a marker if the condition \eqref{eq_stopping} of our stopping criteria of Remark~\ref{rem_stopping} is satisfied.
}\label{fig6}
\end{figure}

The observations (a)--(f) obtained in Example~\ref{Ex5.1} for a particular matrix can also be made for other symmetric positive definite Toeplitz matrices, as long as the condition number remains modest, say, below $10$. Notice that, with $\alpha=-\infty$ and $\beta=0$, the condition number of $A$ equals the cross ratio in \eqref{eq_choice_T} (for the cases $(i)$, $(iii)$ and $(iv)$), and hence determines the asymptotic rate of convergence \eqref{optimal_nodes2}, essentially the slope of our a priori upper bound.

\begin{figure}[!h]
	\begin{center}
		\includegraphics[width=1.03\linewidth]{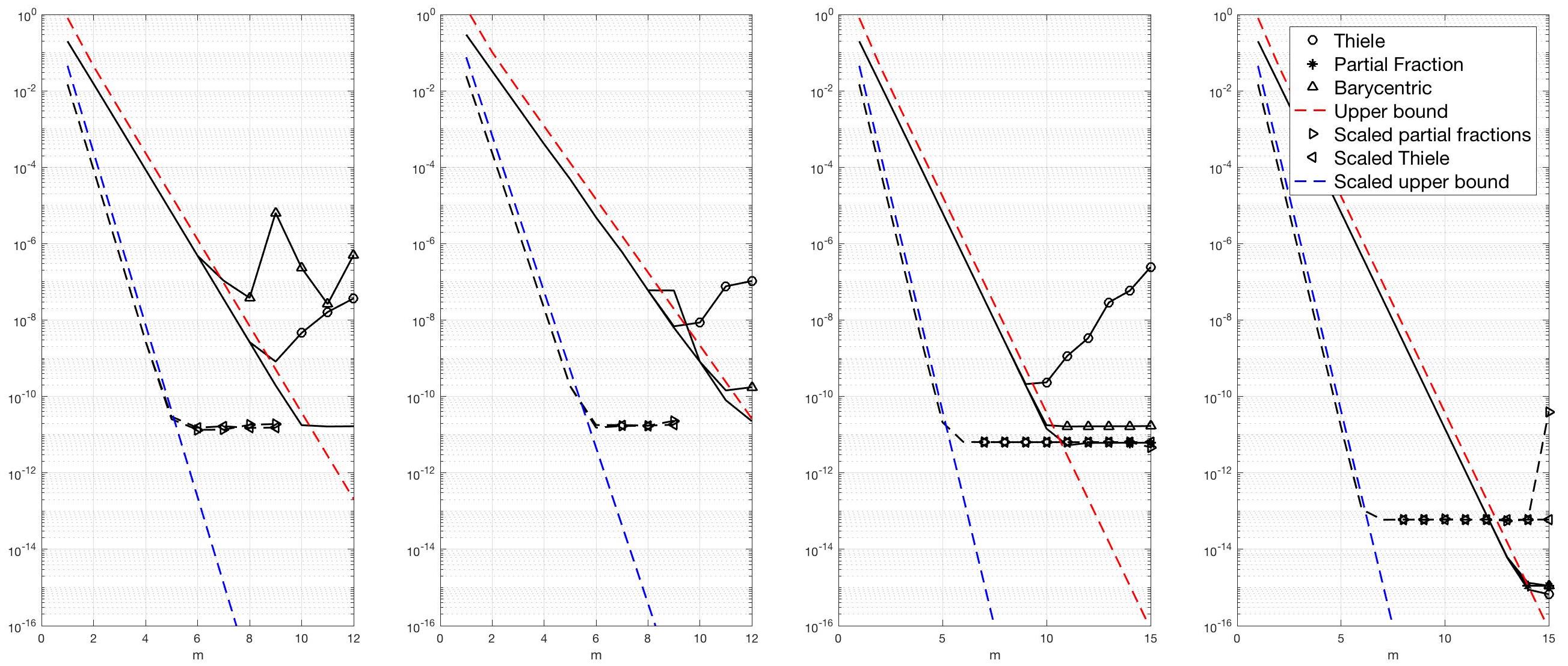}
		\\~\hspace{+0.0cm}
        \includegraphics[width=1.023\linewidth]{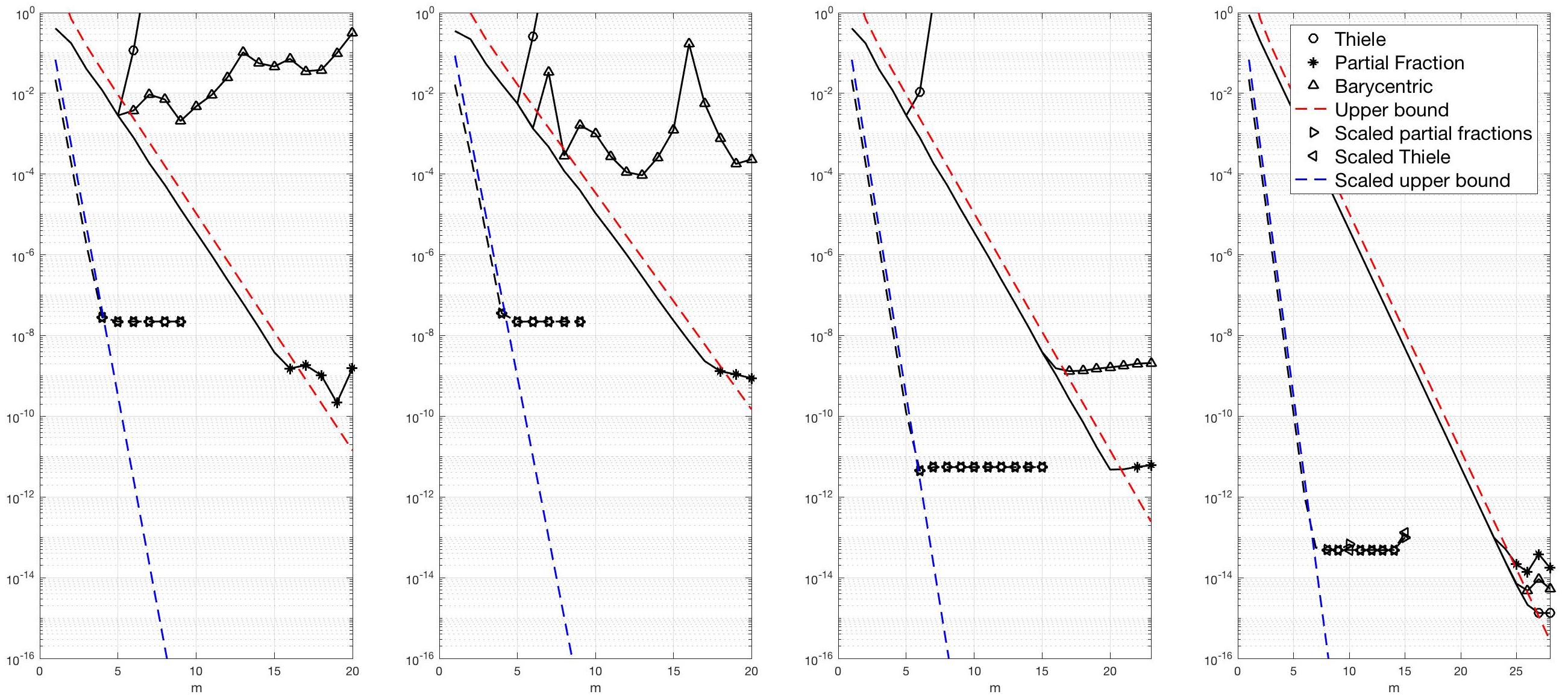}
    \end{center}
		\caption{Relative errors for approaching $f(A)$ for the Markov function $f(z)=\log(z)/(z-1)$, and two symmetric positive definite Toeplitz matrix of order $500$, with extremal eigenvalues $\lambda_{\min}=0.918,\lambda_{\max}=111.7$ and condition number $121.7$ for the first matrix (on the top), and $\lambda_{\min}=0.001,\lambda_{\max}=135$ and condition number $1.35 \cdot 10^5$ for the second matrix (on the bottom).
%We present four different experiments (displayed from the left to the right) denoted by $(i)-(iv)$, and  explained in \eqref{4experiences}. For each experiment, we display (black solid lines) the relative error at our matrix argument of the rational interpolants $r_m$ as a function of the denominator degree $m$, and interpolation points \eqref{optimal_nodes}. Here we show the results of the three approaches of \S 3, namely partial fraction decomposition (marked by stars), the barycentric representation (marked by triangles), and the Thiele continued fraction (marked by circles). We also displayed (red dashed line) the a priori upper bound \eqref{optimal_nodes2}. Notice that the error curves are nearly identical as long as they are below the a priori bound. On each error curve, the error is displayed by a marker if the condition \eqref{eq_stopping} of our stopping criteria of Remark~\ref{rem_stopping} is satisfied.
}\label{fig7}
\end{figure}

%  \begin{figure}
%  \begin{tabbing}
%  ~~~~ ~~~~  \= Given symmetric matrix $B$ with spectrum in $[c,d]$ and parameters $\mu_0,\mu_1,...$
%  \\ \> Initialize $M_0 = X_0 = B$ ,
%  \\ \> For \= $k=0,1,2,...$ until $\| I- M_k \| < tol$ :
%  \\ \> \> $M_{k+1} = \frac{1}{4} \Bigl( 2I + \mu_k^2 M_k + \frac{1}{\mu_k^2} M_k^{-1} \Bigr)$,
%           $X_{k+1}=\frac{1}{2} \mu_k (I + \mu_k^{-2} M_k^{-1}) X_k$.
%  \\ \> Return $X_k$ approximation of $B^{1/2}$
%  \end{tabbing}
%  \caption{Product form of the scaled DB iteration for approximating the matrix square root $B^{1/2}$.}\label{fig_Newton}
%  \end{figure}

However, observation (d) is no longer true for condition numbers larger than $10$: we present in Figure~\ref{fig7} two other examples for $\log(A)$ and two symmetric positive definite Toeplitz matrices $A$, with condition number $121.7$ (on the top) and $1.35 \cdot 10^5$ (on the bottom).\footnote{Further explanations on Figure~\ref{fig7} are given in Example~\ref{Ex5.2} below.} Here not only the barycentric representation but also the Thiele continued fraction approach fails completely to give acceptable relative errors, in particular for the two cases $(i)$ and $(ii)$ of Toeplitz-like arithmetic. In contrast, the partial fraction decomposition gives a relative error of order $10^{-11}$ for the top case, and $10^{-9}$ for the bottom case, which probably is acceptable for most applications.

An alternative to deal with ill-conditioned matrices $A$ is the classical inverse scaling and squaring technique, which can be applied for the logarithm and for the fractional power $x \mapsto x^\gamma$, $\gamma\in \mathbb R$,
  \begin{equation} \label{eq_sas}
        \log(A) = 2^\ell \log ( A^{\frac{1}{2^\ell}} ) , \quad
        A^\gamma = \Bigl( ( A^{\frac{1}{2^\ell}} )^\gamma \Bigr)^{2^\ell},
\end{equation}
see for instance \cite[Chapter 11.5]{higham2008} and \cite{Higham13}. In other words, we start by computing $\ell$ square roots $A_0=A$, and $A_j=(A_{j-1})^{1/2}$ for $j=1,...,\ell$, with the integer $\ell$ chosen such that
\begin{equation} \label{choice_ell}
  \mbox{cond}(A^{\frac{1}{2^\ell}})\leq (d/c)^{\frac{1}{2^\ell}} \leq 10.
\end{equation}
Here each square root is obtained by a scaled Newton method, and more precisely the product form of the scaled DB iteration\footnote{In the original formulation \cite{Denman76}, Denman and Beavers gave a coupled two-term recurrence for $X_k$ and $Y_k:=A^{-1} X_k$. The product form was obtained later in \cite{Cheng01}, where $Y_k$ is replaced by $M_k=X_kY_k$. These authors suggest the recurrence $X_{k+1}=\frac{1}{2} \mu_k X_k(I + \mu_k^{-2} M_k^{-1})$, but our recurrence seems to be more suitable in case where the matrices do no longer commute due to finite precision arithmetic. It seems that this slight modification has no impact on the analysis of stability and limiting accuracy, and even hardly no impact on the (relative) error.} given in Algorithm~\ref{fig_Newton}.
\begin{algorithm}[h!]
\SetAlgoLined
\KwResult{Return $X_k$ approximation of $B^{1/2}$}
\Begin{ $M_0 = X_0 = B$, $k=0$\;
 \While{$\| I- M_k \| > tol$}{
  $M_{k+1} = \frac{1}{4} \Bigl( 2I + \mu_k^2 M_k + \frac{1}{\mu_k^2} M_k^{-1} \Bigr)$\;
  $X_{k+1}=\frac{1}{2} \mu_k (I + \mu_k^{-2} M_k^{-1}) X_k$\;
  $k \leftarrow k+1$\;
 }}
 \caption{Product form of the scaled DB iteration for approximating the matrix square root $B^{1/2}$ for a symmetric positive definite matrix $B$ with spectrum in $[c,d]$ and parameters $\mu_0,\mu_1,...$.}\label{fig_Newton}
\end{algorithm}
Notice that $M_k-I=X_k B^{-1}X_k - I$ is what we have called in Corollary~\ref{cor_posteriori} the residual of the square root $B^{1/2}$. A suitable choice of parameters allows to speed up the first iterations of the Newton method. The following parameters have been suggested for scalar arguments by Rutish\"auser \cite{Rut63}, and discussed for matrix arguments by Beckermann \cite{beckermann2}, see also Zietak \& Zielinski \cite{ZielZietak07} and Byers \& Xu \cite{ByersRalph08},
  \begin{equation} \label{Newton_optimal}
        \mu_0 = \frac{1}{\sqrt[4]{c d}} , \quad \mu_1
        = \sqrt{\frac{2 \sqrt[4]{c d}}{\sqrt{c}+\sqrt{d}}}, \quad
        \mu_{k+1} = \sqrt{\frac{2\mu_k}{1+\mu_k^2}}
  \end{equation}
  for $k \geq 1$. For this choice of parameters (which are $\in(0,1)$ and tend quickly to $1$), one shows by recurrence that, for $k \geq 1$,
  $$
         \sigma(X_k B^{-1/2}) \subset [1,\frac{1}{\mu_k^2}], \quad \sigma(M_k) \subset [1,\frac{1}{\mu_k^4}] , \quad \| M_k - I\| \leq \frac{1-\mu_k^4}{\mu_k^4}.
  $$
  In order to keep stability and limiting accuracy shown for parameters $\mu_k=1$ in \cite[\S 6.4]{higham2008}, we suggest to proceed in two phases: in the first phase we apply Newton with parameters as in \eqref{Newton_optimal} until $\frac{1-\mu_K^4}{\mu_K^4}\leq 10^{-3}$ (for instance $K \leq 5$ for $\mbox{cond}(A)\leq 10^6$). For $k \geq K$, we then choose $\mu_k=1$ and thus $$
      \| M_{k+1} - I \| \leq \frac{1}{4} \| (M_{k} - I)^{2} M_k^{-1} \|
      \leq \frac{1}{3} \| M_{k} - I \|^{2}.
  $$ According to this quadratic convergence, $3$ Newton steps in the second phase should lead to high precision even in finite precision. After having computed $A^{\frac{1}{2^\ell}}$, we then evaluate rational interpolants of our particular Markov function $f$ at $A^{\frac{1}{2^\ell}}$, and finally perform the squaring or renormalization in order to approximate $f(A)$. In the cases $(i)-(ii)$ we implemented Newton and the squaring within the algebra of Toeplitz-like matrices, in order to speed up
  computation time. Notice that the cost of evaluating the interpolants for various values of $m$ is much higher than the cost for scaling or squaring, at least for $\mbox{cond}(A)\leq 10^{6}$ where we have to compute $\ell\leq 3$ square roots, and we have at most $8$ Newton steps for each square root.

\begin{Ex}
\label{Ex5.2}
   Reconsider the problem of approaching $\log(A)$ for the two symmetric positive definite Toeplitz matrices $A$ of Figure~\ref{fig7}. Beside the error curves described in the first paragraph of \S\ref{sec_experiences}, we have added in Figure~\ref{fig7} the a priori upper bound for the matrix $A^{\frac{1}{2^\ell}}$ (in blue dashed), as well as the relative error (black dashed) obtained by evaluating at $A^{\frac{1}{2^\ell}}$ the interpolant of $f(x)=\log(x)/(x-1)$ via a partial fraction decomposition (triangles pointing to the right), or via a Thiele continued fraction (triangles pointing to the left). As we have seen before, both approaches have a very similar behavior according to \eqref{choice_ell}. On the top, with a matrix of condition number $121.7$ (and hence $\ell=2$), we observe that this inverse scaling and squaring technique combined with Toeplitz-like arithmetic gives about the same relative error as partial fraction decomposition applied directly to $A$.
   On the bottom, with a matrix of condition number $1.35 \cdot 10^5$ and hence $\ell=3$, the conclusion is different: here our inverse scaling and squaring technique combined with Toeplitz-like arithmetic gives a relative error about 10 times larger than that for partial fraction decomposition applied directly to $A$.

   Detailed information about the rate of convergence and the final precision of each Newton iteration for computing $A_{j}=(A_{j-1})^{1/2}$ for $j=1,...,\ell$, $A_0=A$ are given in \cite{these_Bisch}, we only report here that the final relative error for scaled Thiele or scaled partial fraction decomposition is dominated by the relative error in computing the square root $A_1=A^{1/2}$, somehow as expected since this matrix has the worst condition number among the matrices $A_j$.
   Also, we tried other equivalent formulations of the Newton method, and obtained similar conclusions.
\end{Ex}

\begin{figure}[!h]
	\begin{center}
		\includegraphics[width=1.03\linewidth]{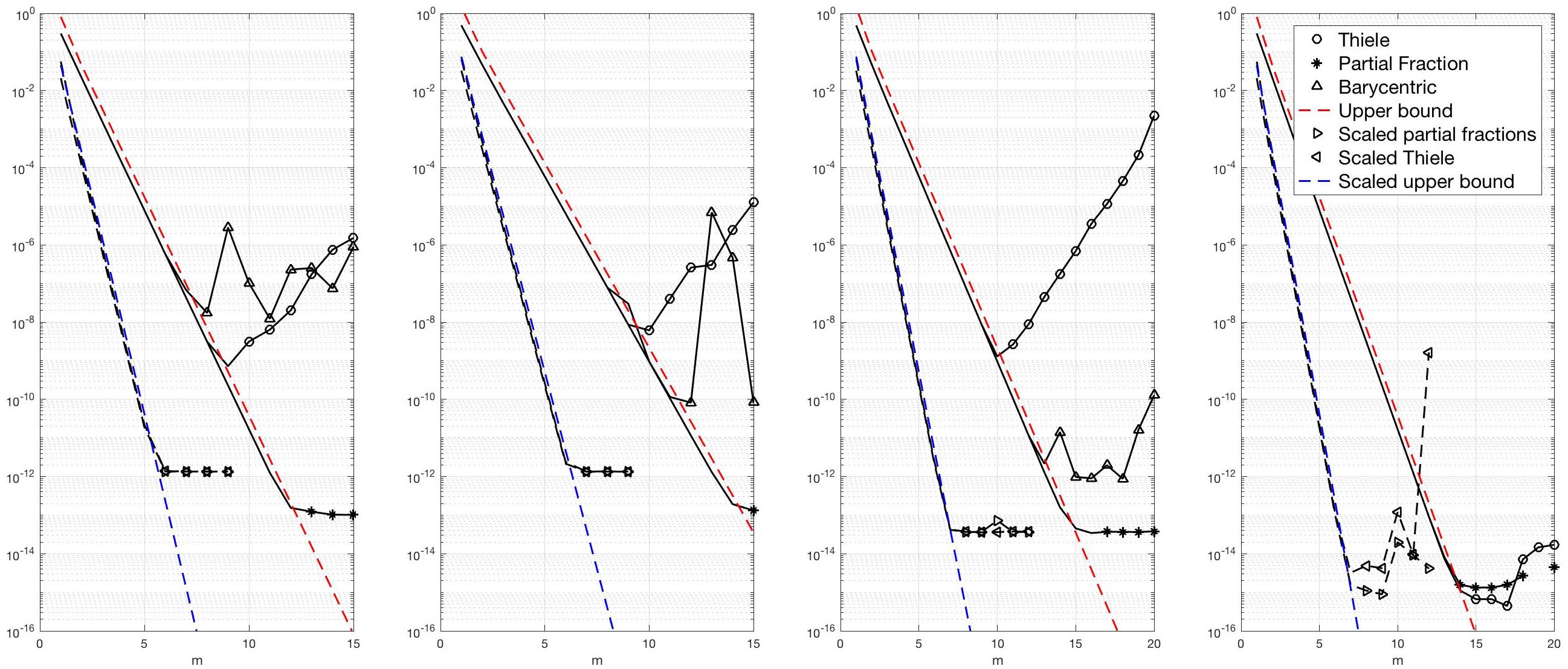}
		\\~\hspace{-0.2cm}
        \includegraphics[width=1.030\linewidth]{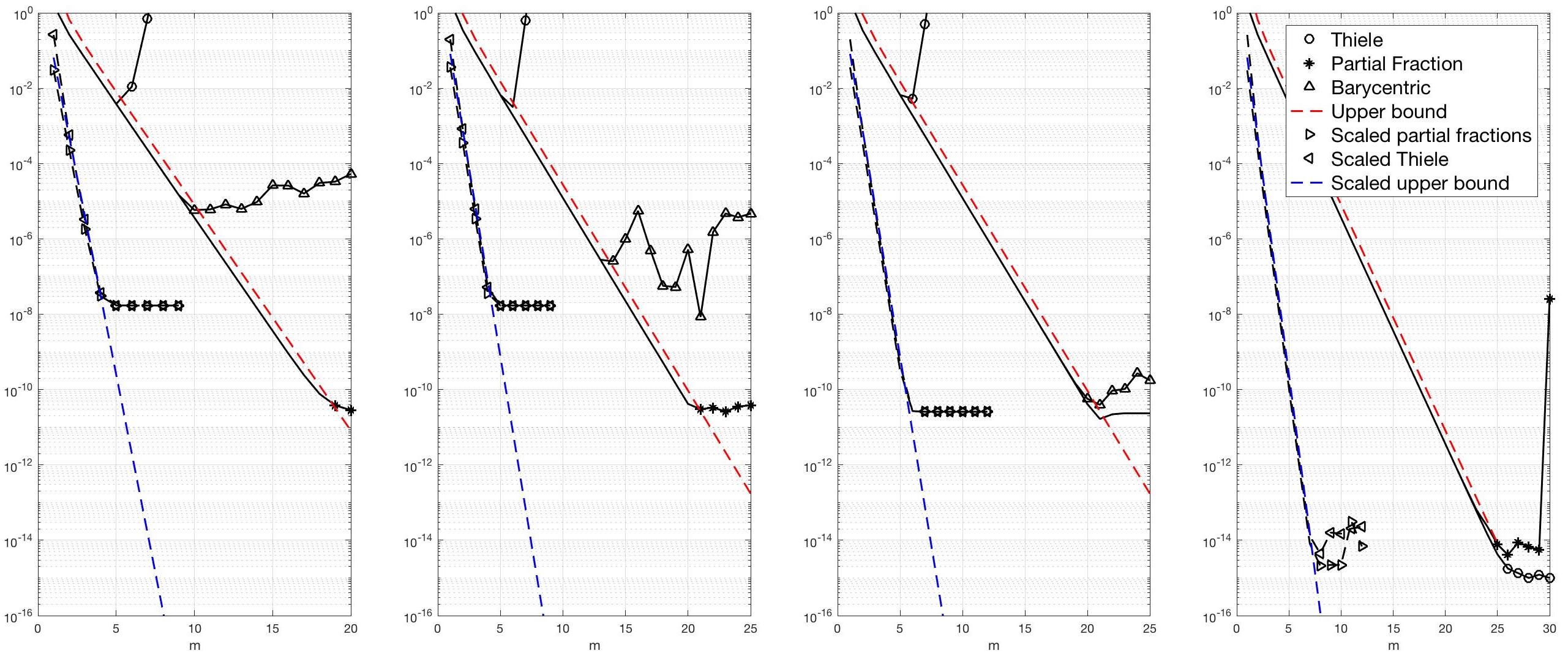}
    \end{center}
		\caption{Relative errors for approaching $A^{-1/3}$ for two symmetric positive definite Toeplitz matrices: on the top we find the same matrix as in Figure~\ref{fig7} of order $500$, with extremal eigenvalues $\lambda_{\min}=0.918,\lambda_{\max}=111.7$ and condition number $121.7$, and on the bottom the $1D$ discretized Laplacian of order $499$ (that is, the tridiagonal matrix containing $2$ on the main diagonal and $-1$ on the super- and subdiagonal), with extremal eigenvalues $\lambda_{\min}=3.95 \cdot 10^{-5},\lambda_{\max}=4.000$ and condition number $1.01 \cdot 10^5$.}\label{fig8}.
\end{figure}

\begin{Ex}
\label{Ex5.3}
  In our final example we study the fractional power $x\mapsto x^{-1/3}$ of two symmetric positive definite Toeplitz matrices, displayed in Figure~\ref{fig8}. We first modify slightly the approach described in \eqref{eq_sas}, since $x \mapsto x^\gamma$ is only a Markov function provided that $\gamma\in [-1,0)$. Also, preliminary numerical experiments not reported here indicate that squaring $\ell$ times seems to increase the relative error. For $\gamma\in \mathbb R$, we thus write $2^\ell \gamma = k + \gamma'$ with $k\in \mathbb Z$ and $\gamma'\in [-1,0)$, such that $A^\gamma=g(A^{\frac{1}{2^\ell}})(A^{\frac{1}{2^\ell}})^k$ with a Markov function $g(x)=x^{\gamma'}$, which is approached by $r_m(A^{\frac{1}{2^\ell}})(A^{\frac{1}{2^\ell}})^k$ with $r_m$ an interpolant of $g$. As in Example~\ref{Ex5.1}, we thus may apply our
  bounds for the relative error.\footnote{It is interesting to compare our findings to those in \cite{Higham13} where the authors approach $A^\gamma$ after scaling by evaluating Padé approximants at the single interpolation point $z_1=...=z_{2m}=1$ using Stieltjes continued fractions, somehow a confluent counterpart of our approach.}

  For the matrix on the top of Figure~\ref{fig8} with condition number $121.7$ we find that $\ell=2$ and hence $k=-1$, $\gamma'=-1/3$, whereas for the matrix on the bottom with condition number $1.01 \cdot 10^5$ we find $\ell=3$ and $k=-2$, $\gamma'=-2/3$.  Limiting ourselves to cases $(i)-(ii)$ using Toeplitz-like arithmetic, we obtain relative errors for unscaled Thiele of about $10^{-7}$ on the top and only $10^{-2}$ on the bottom. Both scaled Thiele or scaled partial fractions allow to achieve relative errors of about $10^{-12}$ on the top and $10^{-8}$ on the bottom.
  However,  the smallest relative error is obtained for an unscaled partial fraction decomposition, namely $10^{-13}$ on the top and $10^{-11}$ on the bottom.
\end{Ex}

%To summarize the numerical results of this section and others reported in \cite{these_Bisch}, for the logarithm and fractional powers, we get small relative errors by using rational interpolants of Markov functions, represented with help of the partial fraction decomposition of scaled

    To summarize, in \S\ref{sec_experiences} we have presented numerical results for two Markov functions and several symmetric positive definite Toeplitz matrices $A$, which show that our stopping criterion of Remark~\ref{rem_stopping} works surprisingly well in practice. Also, exploiting the Toeplitz structure gives an interesting complexity for large $n$, but in general also increases the error. If we exploit the Toeplitz structure, we should avoid the barycentric representation of \S\ref{sec3.2} and the Thiele interpolating continued fraction of \S\ref{sec3.3}, since the smallest relative error is obtained by the partial fraction decomposition of \S\ref{sec3.1}, especially for larger condition numbers of $A$. Finally, for the functions considered in \eqref{eq_sas}, one might also want to combine the partial fraction decomposition of \S\ref{sec3.1} with inverse scaling and squaring, which seems to increase the error, but has lower complexity since the involved rational functions have lower degree.

\section{Conclusion}\label{sec6}

    In this paper we presented a detailed study of how to efficiently and reliably approximate $f(A)$ by $r_m(A)$, with $f$ a Markov function, $A$ a symmetric Toeplitz matrix, and $r_m$ a suitable rational interpolant of $f$. Numerical evidence provided in Figure~\ref{fig_representation1} and case $(iv)$ on the right of Figures~\ref{fig6}--\ref{fig8} shows that, for scalar arguments $z$, we may nearly reach machine precision for the relative error using any of these three approaches discussed in \S\ref{sec3}. The picture changes however completely for the relative error $I-r_m(A) f(A)^{-1}$ evaluated at a Toeplitz matrix argument $A$. Here only the partial fraction decomposition of \S\ref{sec3.1} insures small errors, especially for larger condition numbers of $A$.

In this paper, we have hardly discussed the case of non necessarily symmetric (Toeplitz) matrices $A$, which is left as open question for further research. As explained in Remark~\ref{rem_FoV}, it is possible to construct rational approximants $r_m$, namely Faber images of rational interpolants, such that $f(A)-r_m(A)$ is bounded by $(1+\sqrt{2})$ times the maximum of $f-r_m$ on the field of values of $A$, which again can be related to the interpolation error of a Markov function on the unit disk. However, we expect such field-of-value estimates for non symmetric matrices $A$ not to be very sharp, and maybe other $K$-spectral sets of $A$ \cite[Section 107.2]{Badea13} would be more suitable. Also, it is not clear for us how to represent the rational function $r_m$, and what kind of stability results to expect for evaluating $r_m$ at a complex scalar argument, or at a general matrix $A$.

    Another direction of further research could be to work with variable precision in our compression procedure of computing the numerical displacement rank, which potentially could lead to a much more efficient implementation.

\bigskip

\noindent{\bf Acknowledgements.} The authors want to thank Stefan G\"uttel, Marcel Schweitzer and Leonid Knizhnerman for carefully reading a draft of this manuscript, and for their useful comments.

\bigskip

\noindent{\bf Conflict of interest.} Partial financial support was received from the Labex CEMPI (ANR-11-LABX-0007-01). The authors declare that they have no conflict of interest. 

\bibliography{bib_article}
\bibliographystyle{alpha}
\end{document}